\RequirePackage{fix-cm}

\documentclass[smallextended]{svjour3}       
\smartqed  

\usepackage[T1]{fontenc}
\usepackage[utf8]{inputenc}

\usepackage{graphicx}

\RequirePackage{algorithm}
\usepackage{algpseudocode}

\usepackage{amsmath}
\usepackage{amssymb}
\usepackage{bm}
\usepackage{bbm}

\usepackage{float}

\usepackage[backref,hidelinks]{hyperref}

\DeclareMathOperator*{\argmax}{arg\,max}
\DeclareMathOperator*{\argmin}{arg\,min}

\usepackage{color}
\definecolor{brandeisblue}{rgb}{0.0, 0.44, 1.0}

\newcommand{\revision}[1]{{ #1 }}

\usepackage{pgf,tikz,pgfplots}
\pgfplotsset{compat=newest}
\usepackage{mathrsfs}
\usetikzlibrary{arrows}
\usetikzlibrary[patterns]
\usepackage{todonotes}

\makeatletter
\let\cl@chapter\undefined
\makeatletter

\usepackage[capitalize,nameinlink,noabbrev]{cleveref}[0.21]

\begin{document}

\title{Robust Bilevel Optimization for Near-Optimal Lower-Level Solutions}
\subtitle{}


\author{Mathieu Besançon \and
        Miguel F. Anjos \and
        Luce Brotcorne
}


\institute{Mathieu Besançon \at
              AIS2T, Zuse Institute Berlin, Germany\\
              \email{besancon@zib.de}           
        \and
           Miguel F. Anjos \at
              School of Mathematics, University of Edinburgh, James Clerk Maxwell Building, Peter Guthrie Tait Road, Edinburgh EH9 3FD, UK, and GERAD, 3000, chemin de la C\^ote-Sainte-Catherine, Montr\'eal (Qu\'ebec), H3T 2A7, Canada
        \and
            Luce Brotcorne \at
            INOCS, INRIA Lille Nord-Europe, 40 avenue Halley, 59650 Villeneuve-d'Ascq, France
}


\date{}

\maketitle

\begin{abstract}
Bilevel optimization problems embed the optimality of a subproblem
as a constraint of another optimization problem.
We introduce the concept of near-optimality robustness for bilevel
optimization, protecting the upper-level solution feasibility from limited deviations
from the optimal solution at the lower level.
General properties and necessary conditions for the existence of solutions are derived
for near-optimal robust versions of general bilevel optimization problems.
A duality-based solution method is defined when the lower level is convex,
leveraging the methodology from the robust and bilevel literature.
Numerical results assess the efficiency of exact and heuristic methods and the impact
of valid inequalities on the solution time.
\end{abstract}

\section{Introduction}
Bilevel optimization problems embed the optimality conditions of a
subproblem into the constraints of another one.
They can faithfully model various decision-making problems such as Stackelberg or
leader-follower games, market equilibria, or pricing and revenue management.
A recent review of methods and applications of bilevel problems is presented in \cite{dempe2018review}.
In the classical bilevel setting, when optimizing its objective function,
the upper level anticipates an optimal reaction of the lower level
to its decisions. However, in many practical cases, the
lower level can make near-optimal decisions, by which we mean
decisions that respect bilevel feasibility but with the corresponding objective value different from the optimal value by up to an additive constant.
An important issue in this setting is the definition of the robustness of the
upper-level decisions with respect to such near-optimal lower-level solutions.

In some engineering applications \cite{dempe2002foundations,stavroulakis1998optimal,nicholls1998optimizing},
the decision-maker optimizes an outcome over a dynamical system (modelled as the lower level).
For stable systems, the rate of change of the state variables decreases as the
system converges towards the minimum of a potential function.
If the system is stopped before reaching the minimum, the designer of the system
would require that the upper-level constraints be feasible for near-optimal lower-level solutions.

In economics and decision theory, the concept of \emph{bounded rationality} \cite{simon1972theories}
or $\varepsilon$-rationality \cite{aumann1997rationality}
defines an economic and behavioural interpretation of a decision-making
process where an agent aims to take \emph{any} solution associated with a
``satisfactory'' objective value instead of the optimal one.

Protecting the upper level from a violation of its constraints by deviations of
the lower level is a form of robust optimization, and it is through these lenses that we will view it in this paper.
Therefore, we use the terms ``near-optimality robustness'' and
``near-optimal robust bilevel problem'' or NRB in the rest of the paper
to refer to this form of robustness.
\\

The introduction of uncertainty and robustness in games
has been approached from different points of view in the literature,
with for instance \cite{Aghassi2006} for the existence of robust counterparts
of Nash equilibria without the knowledge of probability distributions associated with the uncertainty.
For uncertainty in hierarchical games and bilevel problems which is our focus in this work,
we refer to the recent survey \cite{beck2022survey} for an overview of approaches and formulations.
In \cite{wardrop07}, the robust version of a network congestion problem is
developed. Users are assumed to make decisions under bounded rationality,
leading to a robust Wardrop equilibrium.
Robust versions of bilevel problems modelling
specific Stackelberg games have been studied in \cite{jain2008robust,pita2010robust},
using robust formulations to protect the leader against non-rationality or
partial rationality of the follower.
A stochastic version of the pessimistic bilevel problem is studied in \cite{yanikoglu2018decision},
where the realization of the random variable occurs after the upper level and before the lower level.
The authors then derive lower and upper bounds on the pessimistic and optimistic
versions of the stochastic bilevel problem as MILPs, leveraging an exact linearization
by assuming the upper-level variables are all binary.
The models developed in \cite{zare2018class} and \cite{zare2020bilevel}
explore different forms of bounded or partial rationality of the lower level,
where the lower level either makes a decision using a heuristic or approximation algorithm or
may deviate from its optimal value in a way that penalizes the objective of the upper level.
In \cite{beck2021robust}, a bilevel model is developed with the lower-level agent facing limit
observability of the upper-level decisions, resulting in another form of uncertainty for the upper level;
the uncertainty is formulated in a fashion inspired by the concept of near-optimality robustness
presented in this work.
In \cite{buchheim2022robust}, a robust version of the bilevel continuous knapsack is considered,
where the upper level does not have a complete knowledge of the lower-level objective function.
They also establish complexity results for the problem which depend on the discrete or continuous
nature of the uncertainty set of the lower-level objective coefficients.\\

Solving bilevel problems under limited deviations of the lower-level response
was introduced in \cite{Wiesemann2013} under the term
``$\varepsilon$-approximation'' of the pessimistic bilevel problem.
The authors focus on the independent case, i.e.~problem settings where the
lower-level feasible set is independent of the upper-level decisions.
Problems in such settings are shown to be simpler to handle than the dependent case
and can be solved in polynomial time when the lower-level problem is
linear under the optimistic and pessimistic assumptions \cite[Theorem 2.2]{Wiesemann2013}.
A custom algorithm is designed for the independent case,
solving a sequence of non-convex non-linear problems relying on global nonlinear solvers.
We must highlight that other bilevel formulations such as the one presented in \cite{10.1007/BFb0083589} have used
concepts of approximate lower level solutions, but that
these formulations are relaxations of the original problem, allowing solutions that are
$\varepsilon$-optimal for the second level and therefore not necessarily bilevel-feasible
unlike our approach or that of \cite{Wiesemann2013}
which \emph{restrict} the feasible space of the standard (both optimistic and pessimistic) bilevel optimization formulation
and robustify the solution.
\\

We consider bilevel problems involving upper- and lower-level variables in the
constraints and objective functions at both levels,
thus more general than the independent ``$\varepsilon$-approximation'' from \cite{Wiesemann2013}.
Unlike the independent case, the dependent bilevel problem is $\mathcal{NP}$-hard even when the
constraints and objectives are linear.
By defining the uncertainty in terms of a deviation from optimality
of the lower level, our formulation offers a novel interpretation
of robustness for bilevel problems and Stackelberg games.
In the case of a linear lower level, we derive an exact MILP reformulation while
not requiring the assumption of pure binary upper-level variables.
\\

The main contributions of the paper are:
\begin{enumerate}
\item The definition and formulation of the dependent near-optimal robust bilevel problem, resulting in a generalized semi-infinite problem and its interpretation as a special case of robust optimization applied to bilevel problems.
\item The study of duality-based reformulations of {NRB} where the lower-level problem is convex conic or linear in \cref{sec:convex}, resulting in a finite-dimensional single-level optimization problem.
\item An extended formulation for the linear-linear {NRB} in \cref{sec:linear}, linearizing the bilinear constraints of the single-level model using disjunctive constraints.
\item Exact and heuristic solution methods for the linear-linear {NRB} in \cref{sec:solalgs} using the extended formulation and its properties.
\end{enumerate}

The paper is organized as follows. In \cref{sec:nodef}, we define the
concepts of near-optimal set and near-optimal robust bilevel problem.
We study the near-optimal bilevel problems with convex and linear lower-level problems in
\cref{sec:convex} and \cref{sec:linear} respectively.
In both cases, the near-optimal robust bilevel problem can be reformulated as a single level.
For a linear lower level, an extended formulation can be derived from the single-level problem.
Exact and heuristic solution algorithms are provided in \cref{sec:solalgs}
and computational experiments are conducted in \cref{sec:numerics},
comparing the extended formulation to the compact one and
studying the impact of valid inequalities.
Finally, in \cref{sec:conclusionNRB} we draw some conclusions and highlight
research perspectives on near-optimality robustness.

\section{Near-optimal set and near-optimal robust bilevel problem}\label{sec:nodef}

In this section, we first define the near-optimal set of the lower level and
near-optimality robustness for bilevel problems.
Next, we illustrate the concepts on an example and highlight several properties
of general near-optimal robust bilevel problems before focusing on the convex
and linear cases in the following sections.

\noindent
The general bilevel problem is classically defined as:
\begin{subequations}\label{prob:bilevelstandard}
\begin{align}
\min_{x}\,\, & F(x,v)\label{eq:basicobj} \\
\text{s.t.} \,\, &  G_k(x,v) \leq 0 & \forall k \in \left[\![m_u\right]\!]\label{eq:upper101} \\
& x \in \mathcal{X}\\
& v \in \argmin_{y \in \mathcal{Y}} \{f(x,y) \text{ s.t. } g_i(x,y) \leq 0\,\forall i \in \left[\![m_l\right]\!] \}.\label{eq:basiclower}
\end{align}
\end{subequations}

The upper- and lower-level objective functions are noted
$F, f : \mathcal{X} \times \mathcal{Y} \rightarrow \mathbb{R}$ respectively.
We denote with $\left[\![n\right]\!]$ the set of running indices $1\dots n$.
$n_U$, $n_L$ denote the number of variables of the upper and lower level respectively,
$m_U$, $m_L$ similarly denote the number of constraints.
Constraint~\eqref{eq:upper101} and $g_i(x,y) \leq 0\,\forall i \in \left[\![m_l\right]\!]$
are the upper- and lower-level constraints respectively.
In this section, we assume that $\mathcal{Y} = \mathbb{R}^{n_l}$
in order that the lower-level feasible set can be only determined by the
$g_i$ functions. The optimal value function $\phi(x)$ is defined as follows:
\begin{align}
& \phi: \mathbb{R}^{n_u} \rightarrow \{-\infty\} \cup \mathbb{R} \cup \{+\infty\}\nonumber \\
& \phi(x) = \min_{y} \{f(x,y) \text{ s.t. } g(x,y) \leq 0\}.\label{eq:optvalfunc}
\end{align}

To keep the notation succinct, the indices of the lower-level constraints $g_i$
are omitted when not needed as in Constraint~\eqref{eq:optvalfunc}.
Throughout the paper, it is assumed that the lower-level problem is feasible
and bounded for any given feasible upper-level decision.\\

\noindent
When, for a feasible upper-level decision, the solution $v$ to the lower-level is not unique, the bilevel problem is not well-defined and further assumptions are
required \cite{dempe2018review}.
In the \emph{optimistic} case, we assume that the lower level selects the optimal solution favouring
the upper level and the optimal solution disfavouring them the most in the \emph{pessimistic} case.
We refer the reader to \cite[Chapter 1]{dempe15} for further details on these two approaches.
\revision{The optimistic case is the most straightforward to formulate using the value function:}
\begin{align}
	\min_{x,v}\,\, & F(x,v) &&\label{prob:optimistic} \tag{B} \\
	\text{s.t.}\,\,& G_k(x,v) \leq 0 && k \in \left[\![m_u\right]\!]\nonumber \\
	& x\in \mathcal{X} && \nonumber\\
	& f(x,v) \leq \phi(x) && \nonumber\\
	& g_i (x,v) \leq 0 && \forall i \in \left[\![m_l\right]\!].\nonumber
\end{align}
\revision{Note that a near-optimal robust problem can be constructed from the original ``unspecified'' problem \eqref{prob:bilevelstandard} or from the optimistic formulation \eqref{prob:optimistic}.}

Instead of making an explicit optimistic/pessimistic assumption about the reaction $v$ of the lower level, we propose a means for the upper level to protect itself against possible follower deviations from its optimality. In other words, we seek a decision $x$ at the upper level that is {\it robust} in the sense that it remains feasible even if the lower level deviates from its own optimality. For this purpose, for a given upper-level decision $x$ and tolerance $\delta$, we define the near-optimal set of the lower level
$\mathcal{Z}(x; \delta)$ as:
\begin{equation*}
\mathcal{Z}(x; \delta) = \{y \,\,|\,\, g(x,y) \leq 0,\, f(x,y) \leq \phi(x) + \delta\}.
\end{equation*}
\noindent
A Near-Optimal Robust Bilevel Problem NRB, of parameter $\delta$ is
defined as a bilevel problem with the additional constraints \eqref{eq:robcons} below ensuring that the upper-level constraints must be satisfied for
any lower-level solution $z$ in the near-optimal set $\mathcal{Z}(x;\delta)$:
\begin{subequations}\label{prob:foundation}
\begin{align}
\text{(NRB)}\,\, \min_{x,v}\,\, & F(x,v)\label{eq:NRBobj}\\
\text{s.t.} \,\, & G_k(x,v) \leq 0 &\forall k \in \left[\![m_u\right]\!]\label{eq:redundant} \\
& f(x,v) \leq \phi(x)\\
& g(x, v) \leq 0\\
& G_k(x,z) \leq 0\,\, \forall z \in \mathcal{Z}(x;\delta) &\forall k \in \left[\![m_u\right]\!]\label{eq:robcons}\\
& x \in \mathcal{X}.\label{eq:foundationlast}
\end{align}
\end{subequations}

\noindent
Each $k$ constraint in \eqref{eq:redundant} is satisfied if the
corresponding constraint set in \eqref{eq:robcons} is non-empty and holds and is therefore redundant
since $v \in \mathcal{Z}(x;\delta)$.
However, we mention Constraint~\eqref{eq:redundant} in the formulation to highlight the structure
of the initial bilevel problem in the near-optimal robust formulation.
\\

NRB captures, unifies, and extends several common formulations
in bilevel optimization.
The special case $\mathcal{Z}(x;0)$ is the set of optimal solutions to the
original lower-level problem, NRB with $\delta=0$ is therefore equivalent to the constraint-based
pessimistic bilevel problem as formulated in \cite{Wiesemann2013}:
\begin{equation*}
	f(x,y) \leq \phi(x) \,\,\forall y \in \mathcal{Z}(x;0).
\end{equation*}
For $\delta<0$, $\mathcal{Z}(x;\delta)$ is the empty set, in which case NRB
is equivalent to the original optimistic bilevel problem.
The set $\mathcal{Z}(x;\infty)$ corresponds to the complete lower-level feasible set,
assuming the lower-level optimal solution is not unbounded for the given upper-level
decision $x$. It therefore results in an optimistic bilevel formulation with a classical
robustness constraints at the upper level.
We also note the connection or near-optimality robustness to the uncertainty model proposed in \cite{zare2020bilevel},
in which the lower-level decision is assumed to be derived from an exact or heuristic
method known from a fixed set of known algorithms. In contrast, we do not make assumptions on a solution process
but consider that the lower-level problem may be solved to near-optimality with a fixed additive constant $\delta$.
This corresponds naturally to several classes of exact algorithms which provide a guarantee on the optimality gap
that depends on the computational effort (e.g.~first-order, interior point methods, branch-and-bound algorithms),
and to all approximation algorithms which would provide solutions with a bound on the optimality gap.
\\

Unlike the constraint-based pessimistic bilevel problem presented in \cite{Wiesemann2013},
the upper-level objective $F(x, v)$ depends on both the upper-
and lower-level variables, but is only evaluated with the optimistic lower-level
variable $v$ and not with a worst-case near-optimal solution.
This modelling choice is enabled by the {NRB} formulation
we chose which uses the optimistic lower-level response $v$ explicitly.
It also implies that the upper level chooses the best optimistic decision
which also protects its feasibility from near-optimal deviations.
One implication for the modeller is that a near-optimal robust problem can be
constructed directly from a bilevel instance where the objective function often
depends on the variables of the two levels, without an epigraph formulation of the objective function.
Alternatively, the near-optimal robust formulation can protect both the upper-level
objective value and constraints from near-optimal deviations of the lower level
using an epigraph formulation introducing an additional variable:
\begin{subequations}\label{prob:nobopconservative}
\begin{align}
\text{(C-NRB)}\,\, \min_{x,v,\tau}\,\, & \tau\label{eq:nobopconservativeobj}\\
\text{s.t.} \,\, & G_k(x,v) \leq 0 & \forall k \in \left[\![m_u\right]\!] \\
& f(x,v) \leq \phi(x)\\
& g(x, v) \leq 0 \\
& F(x, z) \leq \tau  & \forall z \in \mathcal{Z}(x;\delta)\\
& G_k(x,z) \leq 0 & \forall z \in \mathcal{Z}(x;\delta) \,\, \forall k \in \left[\![m_u\right]\!]\\
& x \in \mathcal{X}.\label{eq:nobopconservativelast}
\end{align}
\end{subequations}
\noindent
The two models define different levels of conservativeness and risk. Indeed:
\begin{equation*}
\text{opt(B)} \leq \text{opt(NRB)} \leq \text{opt(C-NRB)},
\end{equation*}
where $\text{opt}(\mathrm{P})$ denotes the optimal value of problem P.
Both near-optimal robust formulations NRB and C-NRB can be of interest to model decision-making applications.
It can also be noted that NRB includes the special
case of interdiction problems, i.e.~problems for which
$F(x, v) = - f(x, v)$.
The two models offer different levels of conservativeness and risk
and can both be of interest when modelling decision-making applications.
\revision{We will focus on Problem~\eqref{prob:foundation}, NRB, in the rest of this article, but all results and formulations extend to Problem~\eqref{prob:nobopconservative}.}
\\

Constraint~\eqref{eq:robcons} is a generalized semi-infinite constraint \cite{gensemiinf}.
The dependence of the set of constraints
$\mathcal{Z}(x;\delta)$ on the decision variables leads to the characterization
of NRB as a robust problem with decision-dependent uncertainty~\cite{Goel2006}.
Each constraint in the set \eqref{eq:robcons} can be replaced by the
corresponding worst-case second-level decision $z_k$ obtained as the solution of
the \emph{adversarial problem}, parameterized by $(x, v, \delta)$:
\begin{subequations}\label{prob:genadv}
\begin{align}
z_k \in \argmax_{y}\,\,& G_k(x,y)\label{eq:advobj} \\
\text{s.t.}\,\, & f(x,y) \leq \phi(x) + \delta \\
& g(x,y) \leq 0.\label{eq:advcons0}
\end{align}
\end{subequations}
The addition of the semi-infinite near-optimal robustness constraint increases the computational difficulty
of the bilevel optimization problem.
Nonetheless, {NRB} and multilevel optimization problems with
similar forms of uncertainty on lower-level decisions do not increase the complexity
of the multilevel problem in the polynomial hierarchy under mild conditions \cite{besanccon2021complexity}.
For bilevel knapsack problems with an uncertain lower-level objective, \cite{buchheim2021complexity} establishes complexity
results in the discrete and continuous lower level cases.

It is common in the robust optimization literature to present models with either uncertainty
on the constraints and/or on the objective function \cite{Gabrel2014}.
As for these, we show that the Objective-Robust Near-Optimal Bilevel Problem (O-NRB),
is a special case of {NRB}:
\begin{align*}
\text{(O-NRB)}\,\, \min_{x \in X} & \sup_{z\in \mathcal{Z}(x;\delta)} F(x,z) \\
\text{s.t.} \,\, & X = \{x \in \mathcal{X}, G_k(x) \leq 0 \,\, \forall k \in \left[\![m_u\right]\!]\}\\
\text{and where: } & \mathcal{Z}(x;\delta) = \{y \,\, | \,\, g(x,y) \leq 0, f(x,y) \leq \phi(x) + \delta\}.
\end{align*}
\noindent
In contrast to most objective-robust problem formulations, the uncertainty
set $\mathcal{Z}$ depends on the upper-level solution $x$, qualifying
O-NRB as a problem with decision-dependent uncertainty.\\

O-NRB is a special case of {NRB}, following a reformulation from objective uncertainty to
constraint uncertainty with an epigraph variable,
O-NRB is equivalent to:
\begin{align*}
\min_{x,\tau}\,\, & \tau \\
\text{s.t.} \,\, &  x \in X\\
& F(x,z) \leq \tau & \forall z \in \mathcal{Z}(x;\delta),
\end{align*}
\revision{this formulation is a special case of {NRB} with an upper-level objective independent of lower-level variables}.
The pessimistic bilevel optimization problem defined in \cite{Loridan1996} is both
a special case and a relaxation of {O-NRB}.
For $\delta=0$, the adversarial problem of
{O-NRB} is equivalent to finding the worst lower-level
decision with respect to the upper-level objective amongst
the lower-level-optimal solutions.
For any $\delta > 0$, the inner problem can select
the worst solutions with respect to the upper-level objective that are not
optimal for the lower level.
The pessimistic bilevel problem is therefore a relaxation of {O-NRB}.\\

\noindent
\revision{We illustrate the concept of near-optimal set and near-optimal robust
solution, first with a simple linear bilevel problemrepresented in \cref{fig:bp1} for a geometric intuition, and then in TODO to show the interest of near-optimality robustness in practical applications.}
\begin{align}
\min_{x,v}\,\, & x \label{prob:example0}\\
\text{s.t.} \,\, & x \geq 0 \nonumber \\
& v \geq 1 - \frac{x}{10}\nonumber \\
& v \in \text{arg} \max_y \{y \text{ s.t. } y \leq 1 + \frac{x}{10} \}.\nonumber
\end{align}
The high-point relaxation of Problem~\eqref{prob:example0}, obtained by relaxing the
optimality constraint of the lower level, while maintaining feasibility, is:
\begin{align*}
\min_{x,v}\,\, & x \\
\text{s.t.} \,\,& x \geq 0 \\
& v \geq 1 - \frac{x}{10}\\
& v \leq 1 + \frac{x}{10}.
\end{align*}
\noindent
The shaded area in \cref{fig:bp1} represents the interior of the polytope, which is
feasible for the high-point relaxation. The induced set, resulting from the
optimal lower-level reaction, is given by:
$\{(x,y) \in (\mathbb{R}_+,\mathbb{R}) \,\text{s.t.}\, y = 1 + \frac{x}{10}\}.$
\noindent
The unique optimal point is $(\hat{x}, \hat{y}) = (0,1)$.

\definecolor{xdxdff}{rgb}{0.49019607843137253,0.49019607843137253,1}
\definecolor{zzttqq}{rgb}{0.6,0.2,0}
\definecolor{ududff}{rgb}{0.30196078431372547,0.30196078431372547,1}

\begin{figure}[H]
\centering
\begin{tikzpicture}[line cap=round,line join=round,>=triangle 45,x=0.6cm,y=0.6cm]
\begin{axis}[x=4cm,y=4cm,axis lines=middle,ymajorgrids=true,xmajorgrids=true,xmin=-0.05,xmax=2.0,ymin=0.45,ymax=1.75,xtick={0,0.5,...,4.5},ytick={0,0.5,...,2.5},]
\clip(-0.3,-0.9) rectangle (4.5,2.7159150117795665);
\fill[line width=2pt,color=zzttqq,fill=zzttqq,fill opacity=0.09] (0,1) -- (5.9649287145746985,0.4035071285425301) -- (5.96565524764239,1.5965655247642392) -- cycle;
\draw [line width=2pt,dash pattern=on 1pt off 1pt,domain=0:4.779549428669664] plot(\x,{(--1--0.1*\x)/1});
\draw [line width=0.7pt,domain=0:4.779549428669664] plot(\x,{(--1-0.1*\x)/1});
\draw [->,line width=2pt] (1.1613278971075351,1.1894718315451782) -- (1.16,1.69);
\draw [->,line width=2pt] (0.89,1.43) -- (0.39,1.43);
\begin{scriptsize}
\draw[color=black] (1.43,1.43) node {f(x) = -y};
\draw[color=black] (0.6847248618593862,1.563278625530058) node {F(x,y) = x};
\draw [fill=ududff] (0,1) circle[radius=2.5pt] ;
\draw[color=black] (0.07,1.17) node {$E$};
\draw [fill=xdxdff] (5.9649287145746985,0.4035071285425301) circle[radius=2.5pt] ;
\draw[color=black] (1.0,0.2) node {x};
\end{scriptsize}
\end{axis}
\end{tikzpicture}
\caption{Linear bilevel problem}
\label{fig:bp1}
\end{figure}
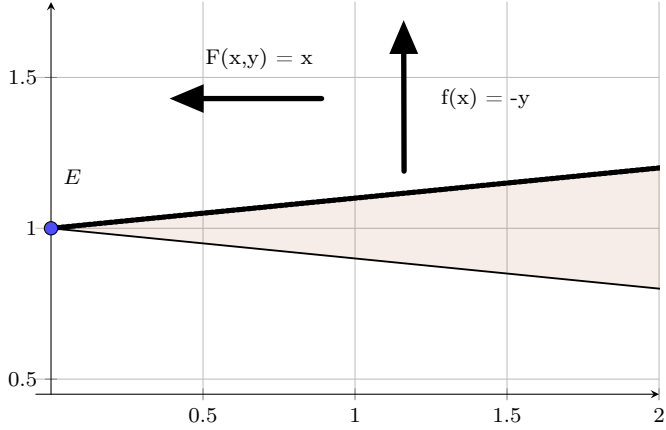

Let us now consider a near-optimal tolerance of the follower with $\delta=0.1$.
If the upper-level decision is $\hat{x}$, then the lower level can take any
value between $1 - \delta = 0.9$ and 1. All these values except 1 lead to an
unsatisfied upper-level constraint problem.
The problem can be reformulated as:
\begin{align*}
\min_{x,v}\,\, & x \\
\text{s.t.} \,\, & x \geq 0 \\
& v \geq 1 - \frac{x}{10}\\
& v \in \text{arg} \max_y \{y \text{ s.t. } y \leq 1 + \frac{x}{10} \} \\
& z \geq 1 - \frac{x}{10} \,\, \forall z \in \{z \,|\, z \leq 1 + \frac{x}{10}, z \geq v - \delta\}.
\end{align*}
\cref{fig:NRB1} illustrates the near-optimal equivalent of the problem with
an additional constraint ensuring the satisfaction of the upper-level constraint
for all near-optimal responses of the lower level.

\begin{figure}[h]
\centering
\begin{tikzpicture}[line cap=round,line join=round,>=triangle 45,x=0.6cm,y=0.6cm]
\begin{axis}[x=4cm,y=4cm,axis lines=middle,ymajorgrids=true,xmajorgrids=true,xmin=-0.05,xmax=2.0,ymin=0.45,ymax=1.75,xtick={-0.1,0.5,...,4.5},ytick={0,0.5,...,2.5},]
\clip(-0.2642017380059678,-0.4660834236385395) rectangle (4.779549428669664,2.7159150117795665);
\fill[line width=2pt,color=zzttqq,fill=zzttqq,fill opacity=0.05] (0,1) -- (5.9649287145746985,0.4035071285425301) -- (5.96565524764239,1.5965655247642392) -- cycle;
\draw [line width=2pt,dash pattern=on 1pt off 1pt,domain=0:4.779549428669664] plot(\x,{(--1--0.1*\x)/1});
\draw [line width=2pt,domain=0:4.779549428669664] plot(\x,{(--1-0.1*\x)/1});
\draw [line width=2pt,dotted,domain=-0.2642017380059678:4.779549428669664] plot(\x,{(--0.9--0.1*\x)/1});
\draw [->,line width=2pt] (1.1613278971075351,1.1894718315451782) -- (1.16,1.69);
\draw [->,line width=2pt] (0.89,1.43) -- (0.39,1.43);
\begin{scriptsize}
\draw[color=black] (1.43,1.43) node {f(x) = -y};
\draw[color=black] (0.6847248618593862,1.563278625530058) node {F(x,y) = x};
\draw [fill=ududff] (0,1) circle[radius=2.5pt] ;
\draw[color=black] (0.07,1.13) node {$E$};
\draw [fill=ududff] (0.5,1.05) circle[radius=2.5pt];
\draw[color=black] (0.43,1.13) node {$F$};
\draw [fill=xdxdff] (5.9649287145746985,0.4035071285425301) circle[radius=2.5pt];
\end{scriptsize}
\end{axis}
\end{tikzpicture}
\caption{Linear bilevel problem with a near-optimality robustness constraint}
\label{fig:NRB1}
\end{figure}
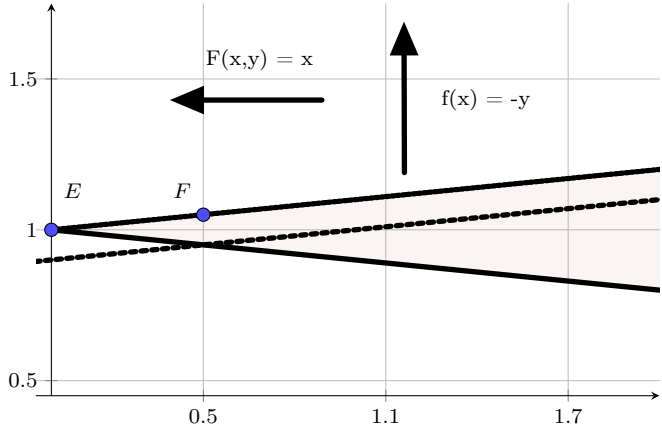

\noindent
This additional constraint is represented by the dashed line.
The optimal upper-level decision is $x = 0.5$,
for which the optimal lower-level reaction is $y = 1+0.1\cdot0.5 = 1.05$.
The boundary of the near-optimal set is $y = 1-0.1\cdot0.5 = 0.95$.\\

\revision{We next present the interpretation of near-optimality robustness in the context of two applications.
In \cite{besanccon2020bileveltlou}, a bilevel model is introduce to determine the parameters of a pricing scheme for electricity in the context of demand response.
The pricing scheme can be summarized as offering users two alternatives for a given time frame, either:
\begin{itemize}
	\item the user remains on the baseline price for the given time frame, in which case the price is ``flat'', i.e.~independent of the consumption,
	\item the user books a certain capacity for a time frame, benefitting from a price lower than the baseline if their consumption remains below the booked capacity, and having to pay a price higher than the baseline if they overconsume.
\end{itemize}
The energy supplier, acting as the leader, seeks to offer a price that incentivizes the user to book a capacity in a certain range, for instance to make the consumption more predictable.
Formulating the problem with the optimistic hypothesis implies that we assume the user will commit to booking a certain capacity and risking to potentially pay a higher price, even if the expected total cost is identical to the simpler flat price.
The pessimistic hypothesis also appears very limited in its mitigation of the problem, since it implies that the user choice preferred by the supplier need to be better than all others by any arbitrarily small gain.
Bounded rationality appears like a more natural representation of the user decision-making process, who would probably not switch to a more complex pricing scheme unless the net gain reaches a certain amount, therefore making near-optimality robustness the natural modelling framework for this problem.\\
Another application in last-mile delivery can be proposed based on the problem introduced in \cite{cerulli2023bilevel}.
In this model, a delivery platform, acting as the leader, must deliver parcels to customers by contracting intermediate delivery carriers acting as followers.
The platform must ensure that the parcels are delivered, and chooses a compensation to offer to each carrier for delivering a particular parcel.
The carriers then select a subset of parcels to deliver maximizing their profit subject to time or space constraints.
It is however natural to assume that some carriers may have constraints or costs that are unknown from the leader, e.g., preferred routes, types of customers, or time windows, but could amount to a difference in objective bounded by some quantity $\delta$.
Near-optimality robustness offers a framework to require solutions ensuring that all the parcels are delivered, even if some carriers choose a near-optimal route.\\

Generalizing from these application, near-optimal robust bilevel optimization is a natural modelling framework for leader-follower games where:
\begin{itemize}
	\item the leader has constraints that depend on the follower decisions,
	\item these constraints cannot be imposed directly to the follower, whose choice can mainly be changed through incentives on their objective,
	\item the user could make a decision with bounded rationality.
\end{itemize}
Typically, when the follower objective amounts to a financial loss or profit, the optimistic and pessimistic hypotheses assume that the follower will adopt a particular behaviour, even if the relative benefit of that behaviour is infinitesimal for them.
In contrast, NRB incorporates a minimum incentive of $\delta$ directly in the model, ensuring the feasibility of the solution for the upper level despite bounded rationality.
Other applications following this structure include pricing problems in networks and shared systems, e.g.~passenger trains and flights, where the price is the only incentive mechanism,
and in which the upper level needs to ensure the system is not used more than its capacity, corresponding to overbooking or overutilization of some resources.
We also note that one group of problems in bilevel optimization that is not suited to NRB is that of interdiction games, since the lower level is already adversarially trying to impact the upper level negatively.
}

In the rest of this section, we establish properties of the near-optimal set and near-optimal
robust bilevel problems.
If the lower-level optimization problem is convex, then
the near-optimal set $\mathcal{Z}(x;\theta)$ is convex as the intersection of two convex sets:
\begin{itemize}
	\item $\{y \,\,|\,\, g(x, y) \leq 0\}$
	\item $\{y \,\,|\,\, f(x, y) \leq \phi(x) + \delta\}$.
\end{itemize}

In robust optimization, the characteristics of the uncertainty set sharply
impact the difficulty of solving the problem.
The near-optimal set of the lower-level is not always bounded;
this can lead to infeasible or ill-defined near-optimal robust counterparts
of bilevel problems. In the next proposition, we define conditions under which
the uncertainty set $\mathcal{Z}(x;\delta)$ is bounded.

\begin{proposition}\label{prob:bounded}
For a given pair $(x,\delta)$, any of the following properties is sufficient
for $\mathcal{Z}(x;\delta)$ to be a bounded set:
\begin{enumerate}
	\item The lower-level feasible domain is bounded.
	\item $f(x,\cdot)$ is radially unbounded with respect to $y$, i.e.\ $\|y\|\rightarrow\infty \Rightarrow f(x, y) \rightarrow \infty$.
	\item $f(x,\cdot)$ is radially bounded such that:
	\begin{equation*}
	\lim_{r \rightarrow +\infty} f(x, r s) > f(x, v) + \delta\,\, \forall s \in \mathcal{S},
	\end{equation*}
	with $\mathcal{S}$ the unit sphere in the space of lower-level variables.
\end{enumerate}
\end{proposition}
\begin{proof}
The first case is trivially satisfied since $\mathcal{Z}(x;\delta)$ is the
intersection of sets including the lower-level feasible set.
If $f(x,\cdot)$ is radially unbounded, for any finite $\delta > 0$,
there is a maximum radius around $v$ beyond which any value of the objective
function is greater than $f(x,v) + \delta$. The third case follows the same
line of reasoning as the second, with a lower bound in any direction
$\| y\| \rightarrow \infty$, such that this lower bound is above $f(x,v) + \delta$.
\qed
\end{proof}

The radius of robust feasibility is defined as the maximum ``size'' of the
uncertain set \cite{liers2019radius,GOBERNA201667}, such that the robust problem
remains feasible. In the case of near-optimality robustness,
the radius can be interpreted as the maximum deviation of the
lower-level objective from its optimal value, such that the near-optimal robust
bilevel problem remains feasible.
\begin{definition}[Radius of near-optimal feasibility]\label{def:tolestim}
\revision{
For a given optimistic bilevel optimization problem \eqref{prob:optimistic} , let $\mathrm{opt}_{\delta}(\mathrm{NRB})$ be the
optimal value of the near-optimal robust problem \eqref{prob:foundation} constructed from \eqref{prob:optimistic} with a tolerance $\delta$.}
The radius of near-optimal feasibility $\hat{\delta}$ is defined by:
\begin{align}
\hat{\delta} = \argmax_{\delta} \,\, \{\delta \text{  s.t. } \mathrm{opt}_{\delta}(\mathrm{NRB}) < \infty\}.
\end{align}
\end{definition}
The radius as defined in \Cref{def:tolestim} can
be interpreted as a maximum robustness budget in terms of the objective value
of the lower level. It represents the maximum level of tolerance of the lower level
on its objective, such that the upper level remains feasible.

\noindent
\begin{proposition}\label{prop:relax}
The optimistic bilevel problem \eqref{prob:optimistic} is a relaxation of
the corresponding near-optimal robust bilevel problem for any $\delta$.
\end{proposition}

\begin{proof}
The optimistic bilevel problem \eqref{prob:optimistic} is equivalent to Problem~\eqref{prob:foundation}
without Constraints \eqref{eq:robcons} and has the same variables as Problem~\eqref{prob:foundation}.
\qed
\end{proof}

\begin{proposition}\label{prop:bileveladv}
If the optimistic bilevel problem \eqref{prob:optimistic} is feasible, then the adversarial problem
\eqref{prob:genadv} is feasible.
\end{proposition}
\begin{proof}
If the bilevel problem is feasible, then the solution $z = v$ is feasible for
the primal adversarial problem.
\qed
\end{proof}

\begin{proposition}
If $(\hat{x},\hat{y})$ is a bilevel-feasible point, and $G_k(\hat{x},\cdot)$
is $K_k$-Lipschitz continuous for a given $k \in \left[\![m_u\right]\!]$
such that:
\begin{align*}
&G_k(\hat{x},\hat{y}) < 0,
\end{align*}
\noindent
then the constraint $G_k(\hat{x}, y) \leq 0$ is satisfied for all $y \in \mathcal{F}_L^{(k)}$  such that:
\begin{equation*}
\mathcal{F}_L^{(k)}(\hat{x}, \hat{y}) = \{y \in \mathbb{R}^{n_l}\,\, | \,\, \|y - \hat{y}\| \leq \frac{|G_k(\hat{x}, \hat{y})|}{K_k} \}.
\end{equation*}
\end{proposition}
\begin{proof}
As $G_k(\hat{x}, \hat{y}) < 0$, and $G_k (\hat{x},\cdot) $ is continuous,
there exists a ball $\mathcal{B}_r(\hat{y})$ in $\mathbb{R}^{n_l}$ centered on
$(\hat{y})$ of radius $r > 0$, such that
\begin{equation*}
G(\hat{x}, y) \leq 0\,\, \forall y \in \mathcal{B}_r(\hat{y}).
\end{equation*}
Let us define:
\begin{align}
r_0 = \argmax_r \,\,\, & \{r\,\,\, \text{s.t.} \,\,\, G(\hat{x}, y) \leq 0 \,\, \forall y \in \mathcal{B}_r(\hat{y})\}. \label{prob:maxR}
\end{align}
By continuity, Problem~\eqref{prob:maxR} always admits a feasible solution. If the feasible set is bounded,
there exists a point $y_0$ on the boundary of the ball, such that $G_k(\hat{x}, y_0) = 0$.
It follows from Lipschitz continuity that:
\begin{align*}
& |G_k(\hat{x}, \hat{y}) - G_k(\hat{x}, y_0)| \leq K_k \|y_0 - \hat{y}\| \\
& \frac{|G_k(\hat{x}, \hat{y})|}{K_k} \leq \|y_0 - \hat{y}\|.
\end{align*}
\noindent
$G_k(\hat{x}, y) \leq G_k(\hat{x}, y_0)\,\, \forall y \in \mathcal{B}_{r_0}(\hat{y})$,
therefore all lower-level solutions in the set
\begin{equation*}
\mathcal{F}_L^{(k)}(\hat{x}, \hat{y}) = \{y \in \mathbb{R}^{n_l}\, \text{s.t.} \, \|y - \hat{y}\| \leq \frac{|G_k(\hat{x}, \hat{y})|}{K_k} \}
\end{equation*}
satisfy the $k$-th constraint.
\qed
\end{proof}

\begin{corollary} Let $(\hat{x}, \hat{y})$ be a bilevel-feasible solution
\revision{to the optimistic bilevel problem \eqref{prob:optimistic}, $\delta$ a tolerance value, and}
\begin{equation*}
\mathcal{F}_L(\hat{x}, \hat{y}) = \bigcap_{k=1}^{m_u} \mathcal{F}_L^{(k)}(\hat{x}, \hat{y}),
\end{equation*}
\noindent
then $\mathcal{Z}(x;\delta) \subseteq \mathcal{F}_L(\hat{x}, \hat{y})$ is a
sufficient condition for near-optimality robustness of $(\hat{x}, \hat{y})$.
\end{corollary}
\begin{proof}
Any lower-level solution $y \in \mathcal{F}_L(\hat{x}, \hat{y})$ satisfies all
$m_u$ upper-level constraints, thus $\mathcal{Z}(x;\delta) \subseteq \mathcal{F}_L(\hat{x}, \hat{y})$
is a sufficient condition for the solution $(\hat{x},\hat{y})$ to be near-optimal robust.
\qed
\end{proof}

\begin{corollary}\label{cor:boundradius}
Let $(\hat{x}, \hat{y})$ be a bilevel-feasible solution
\revision{to the optimistic bilevel problem \eqref{prob:optimistic}, $\delta$ a tolerance value,}
let $R$ be the radius of the lower-level feasible set and $G_k(\hat{x},\cdot)$ be $K_k$-Lipschitz
for a given $k$, then the $k$-th constraint is robust against near-optimal deviations
if:
\begin{equation*}
|G_k(\hat{x}, \hat{y})| \leq K_k R.
\end{equation*}
\end{corollary}
\begin{proof}
The inequality can be deduced from the fact that $\|y - \hat{y}\| \leq R$.
\qed
\end{proof}

\noindent
\Cref{cor:boundradius} can be used when the lower level feasible set is bounded to
verify near-optimality robustness of incumbent solutions. 

\section{Near-optimal robust bilevel problems with a convex lower level}\label{sec:convex}

In this section, we study near-optimal robust bilevel problems where the lower-level
Problem~\eqref{eq:basiclower} is a parametric convex optimization problem with
both a differentiable objective function and differentiable constraints.
If Slater's constraint qualification holds, the KKT
conditions are necessary and sufficient for the optimality of the lower-level
problem and strong duality holds for the adversarial subproblems.
These two properties are leveraged to reformulate {NRB} as a
single-level closed-form problem.\\

\noindent
\revision{Given a pair $(x,v)$, the adversarial
problem associated with the $k$-th constraint of Problem~\eqref{prob:foundation}} can be formulated as:
\begin{subequations}\label{prob:convadvill}
\begin{align}
\max_{y}\,\,\, & G_k(x,y)\label{eq:objadvconv}\\
\text{s.t.} \,\,\,& g(x,y) \leq 0\\
& f(x,y) \leq f(x,v) + \delta.\label{eq:consadvconv}
\end{align}
\end{subequations}
\noindent
Even if the upper-level constraints are convex with respect to $y$,
Problem~\eqref{prob:convadvill} is in general non-convex since the function to
maximize is convex over a convex set. First-order optimality
conditions may induce several non-optimal critical points and the definition of a
solution method needs to rely on global optimization techniques
\cite{globalconcave86,benson1991}.\\

By assuming that the constraints of the upper-level problem $G_k(x,y)$ can be
decomposed and that the projection onto the lower variable space is affine, the upper-level
constraint can be re-written as:
\begin{equation}
G_k(x,y) \leq 0 \Leftrightarrow G_{k}(x) + H_k^{T} y \leq q_k.
\end{equation}
The $k$-th adversarial problem is then expressed as:
\begin{subequations}\label{prob:convadv}
\begin{align}
\max_{y}\,\, & \langle H_k, y\rangle\label{eq:convadvobj}&&\\
\text{s.t.} \,\,\,& g_i(x,y) \leq 0 \,\,\,\,\forall i \in \left[\![m_l\right]\!] \,\, &&(\alpha_i)\\
& f(x,y) \leq f(x,v) + \delta &&(\beta)\label{eq:convadvnear}
\end{align}
\end{subequations}
\noindent
and is convex for a fixed pair $(x, v)$.
Satisfying the upper-level constraint in the worst-case requires that the objective
value of Problem~\eqref{prob:convadv} is lower than $q_k - G_k(x)$.
We denote by $\mathcal{A}_k$ and $\mathcal{D}_k$ the objective values of the
adversarial Problem~\eqref{prob:convadv} and its dual respectively.
$\mathcal{D}_k$ takes values in the extended real set to
account for infeasible and unbounded cases.
\cref{prop:bileveladv} holds for Problem~\eqref{prob:convadv}. The feasibility of the
upper-level constraint with the dual adversarial objective value as formulated
in Constraint~\eqref{eq:dualfeas} is, by weak duality of convex problems, a sufficient condition for the feasibility
of a near-optimal solution. If Slater's constraint qualifications hold,
it is also a necessary condition \cite{boyd2004convex} by strong duality:
\begin{equation}\label{eq:dualfeas}
\mathcal{A}_k \leq \mathcal{D}_k \leq q_k - G_k(x).
\end{equation}

\noindent
The generic form for the single-level reformulation of the near-optimal robust
problem can then be expressed as:
\begin{subequations}\label{prob:convextrans1}
\begin{align}
\min_{x,v} \,\,\,& F(x,v) \\
\text{s.t.} \,\,\, & G(x) + H v \leq q \\
& f(x,v) \leq \phi(x)\label{eq:optvalfuncconv} \\
& g(x,v) \leq 0 \label{eq:feaslowconv} \\
& \mathcal{D}_k \leq q_k - G_k(x) & \forall k \in \left[\![m_u\right]\!]\\
& x \in \mathcal{X}.
\end{align}
\end{subequations}

\noindent
In order to write Problem~\eqref{prob:convextrans1} in a closed form,
the lower-level problem (\ref{eq:optvalfuncconv}-\ref{eq:feaslowconv}) is
reduced to its KKT conditions:
\begin{subequations}
\begin{align}
&\nabla_v f(x,v) - \sum_{i=1}^{m_l} \lambda_i \nabla_v g_i(x,v) = 0 &\\
&g_i(x,v) \leq 0 &\forall i \in \left[\![m_l\right]\!]\\
&\lambda_i \geq 0 &\forall i \in \left[\![m_l\right]\!]\\
&\lambda_i g_i(x,v) = 0 &\forall i \in \left[\![m_l\right]\!]\label{eq:conveq}&.
\end{align}
\end{subequations}

\noindent
Constraint~\eqref{eq:conveq} derived from the KKT conditions
cannot be tackled directly by non-linear solvers \cite{dempe2019kkt}.
Specific reformulations,
such as relaxations of the equality Constraints \eqref{eq:conveq} into inequalities
or branching on combinations of variables (as developed in \cite{scholtes2001convergence,Schewe2019})
are often used in practice.
\\

\noindent
In the rest of this section, we focus on bilevel problems such that the lower
level is a conic convex optimization problem. Unlike the convex version developed
above, the dual of a conic optimization problem can be written in closed form.
\begin{align}
\min_{y}\,\,\, & \langle d, y\rangle\label{eq:confollowerobj} \\
\text{s.t.} \,\,\,& A x + B y = b\nonumber \\
& y \in \mathcal{K}\nonumber
\end{align}

\noindent
where $\langle \cdot, \cdot \rangle$ is the inner product associated with
the space of the lower-level variables
and $\mathcal{K}$ is a proper cone \cite[Chapter~2]{boyd2004convex}.
This class of problems encompasses a broad class of convex optimization problems of practical
interest \cite[Chapter~4]{nesterov1994interior},
while the dual problem can be written in a closed-form if the dual cone is known,
leading to a closed-form single-level reformulation.
The $k-$th adversarial problem is given by:
\begin{subequations}\label{prob:advconv}
\begin{align}
\max_{y,r}\,\,\, & \langle H_k, y\rangle \\
\text{s.t.} \,\,\, & B y = b - A x \\
& \langle d, y\rangle + r = \langle d, v\rangle + \delta\\
& y \in \mathcal{K} \\
& r \geq 0
\end{align}
\end{subequations}
\noindent
with the introduction of a slack variable $r$. With the following change of variables:
\[
\hat{y} =
\begin{bmatrix}
y \\
r
\end{bmatrix}
\,\, \hat{B} =
\begin{bmatrix}
B & 0
\end{bmatrix}
\,\, \hat{d} =
\begin{bmatrix}
d & 1
\end{bmatrix}
\,\, \hat{H}_k =
\begin{bmatrix}
H_k \\
0
\end{bmatrix},
\]
\begin{equation*}
\hat{\mathcal{K}} = \{(y, r),\,\, y \in \mathcal{K},\,\, r \geq 0\},
\end{equation*}
\noindent
$\hat{\mathcal{K}}$ is a cone as the Cartesian product of $\mathcal{K}$ and the
nonnegative orthant. Problem~\eqref{prob:advconv} is reformulated as:
\begin{align*}
\max_{\hat{y}}\,\,\, & \langle\hat{H}_k, \hat{y}\rangle\\
\text{s.t.} \,\,\, & (\hat{B} \hat{y})_i = b_i - (A x)_i \,\,\,\,\forall i \in \left[\![m_l\right]\!] && (\alpha_i)\\
& \langle \hat{d}, \hat{y} \rangle = \langle d, v\rangle + \delta && (\beta)\\
& \hat{y} \in \hat{\mathcal{K}}
\end{align*}

\noindent
which is a conic optimization problem, for which the dual problem is:
\begin{subequations}\label{prob:dualconv1}
\begin{align}
\min_{\alpha,\beta,s_k} \,\,\,& \langle (b - A x), \alpha\rangle + ( \langle d, v\rangle + \delta) \beta \\
\text{s.t.} \,\,\, & \hat{B}^\top \alpha + \beta \hat{d} + s = \hat{H}_k \\
& s \in -\hat{\mathcal{K}}^*,
\end{align}
\end{subequations}

\noindent
with $\hat{\mathcal{K}}^*$ the dual cone of $\hat{\mathcal{K}}$.
In the worst case (maximum number of non-zero coefficients), there are
$(m_l \cdot n_u + n_l)$ bilinear terms in $m_u$ non-linear non-convex constraints.
This number of bilinear terms can be reduced by introducing the following
variables $(p,o)$, along with the corresponding constraints:
\begin{subequations}\label{prob:dualconv2}
\begin{align}
\min_{\alpha,\beta,s,p,o}\,\,\,& \langle p, \alpha\rangle + (o + \delta) \beta\label{eq:convbilinear} \\
\text{s.t.} \,\,\,& p = b - A x \\
& o = \langle d, v \rangle \\
& \hat{B}^\top \alpha + \beta \hat{d} + s = \hat{H}_k \\
& s \in -\hat{\mathcal{K}}^*.
\end{align}
\end{subequations}

\noindent
The number of bilinear terms in the set of constraints is thus reduced
from $n_u  m_l + n_l$ to $m_l + 1$ terms in \eqref{eq:convbilinear}.
Problem~\eqref{prob:dualconv1} or equivalently Problem~\eqref{prob:dualconv2} have
a convex feasible set but a bilinear non-convex objective function.
The KKT conditions of the follower Problem~\eqref{eq:confollowerobj}
are given for the primal-dual pair $(y,\lambda)$:
\begin{subequations}
\begin{align}
& B y = b - A x \\
& y \in \mathcal{K} \\
& d - B^\top \lambda \in \mathcal{K}^*\\
& \langle d - B^\top \lambda, y \rangle = 0.
\end{align}
\end{subequations}

\noindent
The single-level problem is:
\begin{subequations}
\begin{align}\label{prob:lincompconic}
\min_{x,v,\lambda,\alpha,\beta,s} \,\,\,& F(x,v) \\
\text{s.t.} \,\,\, & G(x) + H v \leq q \\
& A x + B v = b\label{eq:confollow1}\\
& d - B^\top \lambda \in \mathcal{K}^*\\
& \langle d - B^\top \lambda, v \rangle = 0\label{eq:confollowlast} \\
& \langle Ax - b, \alpha_{k}\rangle + \beta_k\, ( \langle v, d\rangle  + \delta) \leq q_k - (G x)_k && \forall k \in \left[\![m_u\right]\!]\label{eq:con1} \\
& \hat{B}^\top \alpha_k + \hat{d} \beta_k + s_k = \hat{H}_k && \forall k \in \left[\![m_u\right]\!] \\
& x \in \mathcal{X}, v \in \mathcal{K}\\
& s_k \in -\hat{\mathcal{K}}^* && \forall k \in \left[\![m_u\right]\!].
\end{align}
\end{subequations}

\noindent
The Mangasarian-Fromovitz constraint qualification is violated at every feasible
point of Problem~\eqref{prob:lincompconic} because of the complementarity
Constraints \eqref{eq:confollowlast} \cite{scheel2000mathematical}.
In non-linear approaches to complementarity constraints
\cite{scholtes2001convergence,dempe2019kkt},
parameterized successive relaxations which respect constraint qualifications are used:
\begin{subequations}\label{cons:eps}
\begin{align}
& \langle d - B^\top \lambda, v\rangle \leq \varepsilon\\
- & \langle d - B^\top \lambda, v\rangle \leq \varepsilon.
\end{align}
\end{subequations}

Constraints \eqref{eq:con1} and \eqref{cons:eps} are both bilinear non-convex
inequalities, the other ones added by the near-optimal robust model
are conic and linear constraints. In general and unlike the complementarity
constraints, the feasible region defined by near-optimal robust constraints
admits strictly feasible solutions and therefore respect constraint qualifications.

\begin{remark}
If $H_k$ belongs to the interior of the polar set of $\mathcal{K}$, and if the sufficient
conditions for applying KKT to the lower level hold, then the k-th dual adversarial
problem admits strictly feasible solutions.
\end{remark}
\begin{proof}
From the assumptions on the lower-level KKT conditions, the dual lower level
admits strictly feasible solutions, i.e.
\begin{equation}\label{eq:strictfeas}
\exists \lambda_0 \in \mathbb{R}^{m_u}, d - B^\top \lambda_0 \in \mathrm{int}(\mathcal{K}^*).
\end{equation}
\noindent
The dual adversarial feasible set is described by the constraints:
\begin{align*}
& B^\top \alpha + \beta d -H_k \in \mathcal{K}^*\\
& \beta \geq 0.
\end{align*}
From \eqref{eq:strictfeas}, by setting $\alpha = -\lambda_0, \beta = 1$, we have
$B^\top \alpha + \beta d \in \mathrm{int}(\mathcal{K}^*)$. $\mathcal{K}^*$ is a closed convex
cone and is therefore closed under addition:
\begin{align*}
& \frac{-H_k}{2} + \frac{1}{2}(B^\top\alpha + \beta d) \in \mathcal{K}^* \,\,\,\Leftrightarrow\\
& 2 (\frac{-H_k}{2} + \frac{1}{2}(B^\top\alpha + \beta d)) \in \mathcal{K}^*.
\end{align*}
\noindent
Since $B^\top\alpha + \beta d \in \mathrm{int}(\mathcal{K}^*)$ is in the interior of
$\mathcal{K}^*$, so is $\frac{1}{2}(B^\top\alpha + \beta d)$.
\begin{equation*}
2 (\frac{-H_k}{2} + \frac{1}{2}(B^\top\alpha + \beta d))
\end{equation*}
then lies on the open segment between a point in $int(\mathcal{K}^*)$ and a
point in the cone (that may or be not be in the interior),
it is therefore an interior point.
\qed
\end{proof}

In conclusion, near-optimal robustness has only added a finite number of
constraints of the same nature (bilinear inequalities)
to the reformulation proposed in \cite{dempe2019kkt}.
Solution methods used for bilevel problems with convex lower-level thus apply
to their near-optimal robust counterpart.

\section{Linear near-optimal robust bilevel problem}\label{sec:linear}

In this section, we focus on near-optimal robust linear-linear
bilevel problems. More precisely, the structure of the lower-level problem is
exploited to derive an extended formulation leading to an efficient solution
algorithm.
We consider that all vector spaces are subspaces of $\mathbb{R}^n$,
with appropriate dimensions. The inner product of two vectors
$\langle a, b \rangle$ is equivalently written $a^\top b$. \\

\noindent
The linear near-optimal robust bilevel problem is formulated as:
\begin{subequations}\label{prob:lin101}
\begin{align}
\min_{x,v}\,\,\, & c_x^\top x + c_y^\top v \\
\text{s.t.}\,\,\, & G x + H v \leq q \\
& d^\top v \leq \phi(x) \\
& A x + B v \leq b\\
& G x + H z \leq q& \forall z \in \mathcal{Z}(x;\delta)\label{eq:linrobustset} \\
& v \in \mathbb{R}^{n_l}_{+}\\
& x \in \mathcal{X}.
\end{align}
\end{subequations}
\noindent
For a given pair $(x,v)$, each semi-infinite robust constraint
\eqref{eq:linrobustset} can be reformulated as the objective value of the
following adversarial problem:
\begin{subequations}\label{prob:linadv}
\begin{align}
\max_{y}\,\,\, & H_k^{T} y\label{eq:robustobj0}\\
\text{s.t.} \,\,\, & (B y)_i \leq b_i - (A x)_i \,\,\, \forall i \in \left[\![m_l\right]\!] & (\alpha_i)\label{eq:robustdualalpha} \\
& d^\top y \leq d^\top v + \delta & (\beta)\label{eq:robustconsdelta} \\
& y \in \mathbb{R}^{n_l}_{+}.&
\end{align}
\end{subequations}
Let $(\alpha,\beta)$ be the dual variables associated with each group of
constraints (\ref{eq:robustdualalpha}-\ref{eq:robustconsdelta}).
The near-optimal robust version of Problem~\eqref{prob:lin101} is feasible
only if the objective value of each $k$-th adversarial subproblem
\eqref{prob:linadv} is lower than $q_k - (G x)_k$.
The dual of Problem~\eqref{prob:linadv} is defined as:
\begin{subequations}\label{prob:dualadvlin}
\begin{align}
\min_{\alpha,\beta}\, & \alpha^\top (b - A x) + \beta\,(d^\top v + \delta) \\
\text{s.t.}\,\, & B^\top\alpha + \beta d \geq H_{k}\label{eq:dualadvlin1} \\
& \alpha \in \mathbb{R}^{m_l}_{+} \beta \in \mathbb{R}_{+}.\label{eq:dualadvlin2}
\end{align}
\end{subequations}

Based on Problem~\eqref{prop:bileveladv} and weak duality results, the dual problem is
either infeasible or feasible and bounded. By strong duality, the objective value
of the dual and primal problems are equal. This value must be smaller than
$q_k - (Gx)_k$ to satisfy Constraint~\eqref{eq:linrobustset}.
This is equivalent to the existence of a feasible dual solution $(\alpha, \beta)$
certifying the feasibility of $(x,v)$ within the near-optimal set
$\mathcal{Z}(x;\delta)$. We obtain one pair of certificates $(\alpha,\beta)$ for each
upper-level constraint in $\left[\![m_u\right]\!]$, resulting in the following problem:
\begin{subequations}\label{eq:naivelin}
\begin{align}
\min_{x,v,\alpha,\beta}\,\,\, & c_x^\top x + c_y^\top v \\
\text{s.t.} \,\,\, & G x + H v \leq q \\
& d^\top v \leq \phi(x) \\
& A x + B v \leq b\\
& \alpha_k^\top (b - A x) + \beta_k\, (d^\top v + \delta) \leq q_k - (G x)_k &&\forall k \in \left[\![m_u\right]\!]\label{eq:colinearcons}\\
& B^\top\alpha_k + \beta_k d \geq H_{k}&&\forall k \in \left[\![m_u\right]\!] \\
& \alpha_k \in \mathbb{R}^{m_l}_{+} \beta_k \in \mathbb{R}_{+} &&\forall k \in \left[\![m_u\right]\!]\\
& v \in \mathbb{R}^{n_l}_{+} \\
& x \in \mathcal{X}.
\end{align}
\end{subequations}
Lower-level optimality is guaranteed by the corresponding KKT conditions:
\begin{subequations}
\begin{align}
& d_j + \sum_i B_{ij} \lambda_i - \sigma_j = 0 && \forall j \in \left[\![n_l\right]\!] \\
& 0 \leq b_i - (Ax)_i - (Bv)_i \,\,\bot\,\, \lambda_i \geq 0 && \forall i \in \left[\![m_l\right]\!] \label{eq:complin1} \\
& 0 \leq v_j \,\,\bot\,\, \sigma_j \geq 0 && \forall j \in \left[\![n_l\right]\!]\label{eq:complin2} \\
& \sigma \geq 0, \lambda \geq 0
\end{align}
\end{subequations}
\noindent
where $\bot$ defines a complementarity constraint.
A common technique to linearize Constraints (\ref{eq:complin1}-\ref{eq:complin2}) is the ``big-M'' reformulation,
introducing auxiliary binary variables with primal and dual upper bounds.
The resulting formulation has a weak continuous relaxation.
Furthermore, a correct choice of bounds is itself an $\mathcal{NP}$-hard
Problem~\cite{kleinert2019there}, and the incorrect choice of these bounds can lead to
cutting valid and potentially optimal solutions \cite{Pineda2018}.
Other modelling and solution approaches, such as special ordered sets of type 1
(SOS1) or indicator constraints avoid the need to specify such bounds
in a branch-and-bound procedure.\\

\noindent
The aggregated formulation of the linear near-optimal robust bilevel problem is:
\begin{subequations}\label{prob:lintotal}
\begin{align}
\min_{x,v,\lambda,\sigma,\alpha,\beta} \,\,\,\,& c_x^\top x + c_y^\top v \\
\text{s.t.} \,\,\,& G x + H v \leq q \\
& A x + B v \leq b \\
& d_j + \sum_i \lambda_i B_{ij} - \sigma_j = 0 &&\forall j \in \left[\![n_l\right]\!]\\
& 0 \leq \lambda_i \,\bot\, A_i x + B_i v - b_i \leq 0&&\forall i \in \left[\![m_l\right]\!]\\
& 0 \leq \sigma_j \,\bot\, v_j \geq 0 &&\forall j \in \left[\![n_l\right]\!]\\
& x \in \mathcal{X}\\
& \alpha_{k}^\top (b - A x) + \beta_k (d^\top v + \delta) \leq q_k - (G x)_k&& \forall k \in \left[\![m_u\right]\!]\label{eq:dualadvfin1}\\
& \sum_{i=1}^{m_l} B_{ij} \alpha_{ki} + \beta_k d_j \geq H_{kj} && \forall k \in \left[\![m_u\right]\!],\ \forall j \in\left[\![n_l\right]\!]\label{eq:dualadvfin12} \\
& \alpha_k \in \mathbb{R}^{m_l}_+, \beta_k \in \mathbb{R}_+&& \forall k \in \left[\![m_u\right]\!].\label{eq:dualadvfin2}
\end{align}
\end{subequations}

Problem~\eqref{prob:lintotal} is a single-level problem and has a closed form.
However, constraints \eqref{eq:dualadvfin1} contain bilinear terms, which
cannot be tackled as efficiently as convex constraints by branch-and-cut based
solvers.
Therefore, we exploit the structure of the dual adversarial problem and its relation
to the primal lower level to design a new efficient reformulation and solution
algorithm.

\subsection{Extended formulation}\label{sub:extended}

The bilinear constraints \eqref{eq:dualadvfin1} involve products of variables from the
upper and lower level $(x,v)$ as well as dual variables of each of the $m_u$
dual adversarial problems.
For the rest of this paper, when the value of a variable $a$ is fixed
in a given problem, we will denote it with $\overline{a}$.
For fixed values $(\overline{x},\overline{v})$ of $(x,v)$, $m_u$ dual adversarial subproblems \eqref{prob:dualadvlin} are
defined. The optimal value of each $k$-th subproblem must be lower
than $q_k - (G\overline{x})_k$ for near-optimality robustness to hold. The feasible region of each subproblem is defined by
(\ref{eq:dualadvfin1}-\ref{eq:dualadvfin2}) and is independent of $(\overline{x},\overline{v})$; their
objective functions are linear in $(\alpha,\beta)$.
Following \cref{prop:bileveladv} and by weak duality, Problem~\eqref{prob:dualadvlin} is bounded.
If, moreover, Problem~\eqref{prob:dualadvlin} is feasible, at least a vertex of the polytope
(\ref{eq:dualadvfin12}  -\ref{eq:dualadvfin2}) is an optimal solution.
Following these observations, Constraints (\ref{eq:dualadvfin1}-\ref{eq:dualadvfin2}) can be replaced
by disjunctive constraints, such that for each $k$, at least one
extreme vertex of the $k$-th dual polyhedron is feasible.
This reformulation of the bilinear constraints has, to the best of our knowledge,
never been developed in the literature.
We must highlight that disjunctive formulations are well established in the bilevel
literature to express the complementarity constraints from the lower-level KKT
conditions \cite{audet2007disjunctive,siddiqui2013sos1,pineda2018efficiently}.
However, the bilinear reformulation of near-optimality robustness constraints
does not possess the same structure and thus cannot leverage similar techniques.\\

Let $\mathcal{V}_k$ be the number of vertices of the
$k$-th subproblem and $\overline{\alpha}^l_k,\overline{\beta}^l_k$ be the $l$-th
vertex of the $k$-th subproblem.
Constraints (\ref{eq:dualadvfin1}-\ref{eq:dualadvfin2}) can be written as:

\begin{align}
& \bigvee_{l=1}^{\mathcal{V}_k} \sum_{i=1}^{m_l} \overline{\alpha}_{ki}^l (b - Ax)_i + \overline{\beta}_k^l \cdot (d^\top v + \delta) \leq q_k - (G x)_k  && \forall k \in \left[\![m_u\right]\!],\label{eq:disjunctive}
\end{align}

\noindent
where $\bigvee_{i=1}^{N} \mathcal{C}_i$ is the disjunction (logical ``OR'') operator,
expressing the constraint that at least one of the constraints $\mathcal{C}_i$
must be satisfied. These disjunctions are equivalent to indicator constraints
\cite{Bonami2015}.\\

This reformulation of bilinear constraints based on the polyhedral description
of the $(\alpha,\beta)$ feasible space is similar to the Benders decomposition\cite{rahmaniani2017benders}.
Indeed, in the near-optimal robust extended formulation, at least one of the vertices
must satisfy a constraint (a disjunction) while Benders decomposition consists in satisfying
a set of constraints for all extreme vertices and rays of the dual polyhedron
(a constraint described with a universal quantifier).
Disjunctive constraints \eqref{eq:disjunctive} are equivalent to the following
formulation, using set cover and SOS1 constraints:
\begin{subequations}\label{prob:sos1}
\begin{align}
& \theta_k^l \in \{0,1\} && \forall k \in \left[\![m_u\right]\!], \forall l\in\mathcal{V}_k \\
& \omega_k^l \geq 0 && \forall k \in \left[\![m_u\right]\!], \forall l\in\mathcal{V}_k \\
& (b-Ax)^\top \overline{\alpha}_k^l + \overline{\beta}_k^l (d^\top v + \delta) - \omega_k^l \leq q_k - (Gx)_k && \forall k \in \left[\![m_u\right]\!], \forall l\in\mathcal{V}_k\\
& \sum_{l=1}^{\mathcal{V}_k} \theta_k^l\geq 1 && \forall k \in \left[\![m_u\right]\!] \\
& SOS1(\theta_k^l, \omega_k^l) && \forall k \in \left[\![m_u\right]\!], \,\, \forall l\in\mathcal{V}_k,
\end{align}
\end{subequations}
where $SOS1(a, b)$ expresses a SOS1-type constraint between the variables $a$ and $b$.
In conclusion, using disjunctive constraints over the extreme vertices of each
dual polyhedron and SOS1 constraints to linearize the complementarity constraints
leads to an equivalent reformulation of Problem~\eqref{prob:lintotal}.
The finite solution property holds even though the boundedness of the
dual feasible set is not required. This single-level extended reformulation
can be solved by any off-the-shelf MILP solver.
We next illustrate the extended formulation on the following example.

\subsection{Bounded example}

Consider the bilevel linear problem defined by the following data:

\begin{equation*}
x \in \mathbb{R}_+, y \in \mathbb{R}_+
\end{equation*}

\[
G =
\begin{bmatrix}
-1 \\
1
\end{bmatrix}
\,\, H =
\begin{bmatrix}
4\\2
\end{bmatrix}
\,\, q =
\begin{bmatrix}
11 \\ 13
\end{bmatrix}
\,\, c_x = \begin{bmatrix} 1 \end{bmatrix}
\,\, c_y = \begin{bmatrix} -10 \end{bmatrix}
\]\\
\[
A =
\begin{bmatrix}
-2 \\
5
\end{bmatrix}
\,\, B =
\begin{bmatrix}
-1\\-4
\end{bmatrix}
\,\, b =
\begin{bmatrix}
-5 \\ 30
\end{bmatrix}
\,\, d =
\begin{bmatrix}
1
\end{bmatrix}.
\]

The optimal solution of the high-point relaxation $(x,v) = (5, 4)$
is not bilevel-feasible. The optimal value of the optimistic bilevel
problem is reached at $(x,v) = (1,3)$.
These two points are respectively represented by the blue diamond and red cross in
\cref{fig:bilevel102}. The dotted segments represent the upper-level constraints
and the solid lines represent the lower-level constraints.

\definecolor{qqwuqq}{rgb}{0,0.39215686274509803,0}
\definecolor{sexdts}{rgb}{0.1803921568627451,0.49019607843137253,0.19607843137254902}
\definecolor{rvwvcq}{rgb}{0.08235294117647059,0.396078431372549,0.7529411764705882}
\definecolor{yqqqqq}{rgb}{0.5019607843137255,0,0}
\definecolor{ffwwqq}{rgb}{1,0.4,0}
\definecolor{wwzzff}{rgb}{0.4,0.6,1}
\begin{figure}[ht]
\begin{minipage}{0.75\textwidth}
\begin{tikzpicture}[line cap=round,line join=round,>=triangle 45,x=0.9cm,y=0.9cm]
\begin{axis}[
x=0.9cm,y=0.9cm,
axis lines=middle,
xmin=-0.5,
xmax=8.5,
ymin=-0.5,
ymax=5.5,
xtick={-1,0,...,14},
ytick={-3,-2,...,7},]
\clip(-1.0937275113352196,-3.747175320934984) rectangle (14.989133833193167,7.767575441284353);
\draw [line width=0.5pt,color=sexdts,domain=-1.0937275113352196:14.989133833193167] plot(\x,{(--3.8--0.1*\x)/1});
\draw [line width=1pt] (2.5,0)-- (1,3);
\draw [line width=1pt] (2.5,0)-- (6,0);
\draw [line width=1pt] (6,0)-- (8,2.5);
\draw [line width=1pt,dash pattern=on 1pt off 2pt] (1,3)-- (5,4);
\draw [line width=1pt,dash pattern=on 1pt off 2pt] (8,2.5)-- (5,4);
\draw [->,line width=0.5pt,color=qqwuqq] (6.065499412011319,4.400266281843231) -- (6,5);
\draw [->,line width=0.5pt,color=black] (1.0,2.0) -- (1.0,1.5);
\begin{footnotesize}
\draw [color=yqqqqq,very thick] (1,3)-- ++(-2.5pt,-2.5pt);
\draw [color=yqqqqq,very thick] (1,3)-- ++(2.5pt,2.5pt);
\draw [color=yqqqqq,very thick] (1,3)-- ++(2.5pt,-2.5pt);
\draw [color=yqqqqq,thick] (1,3)-- ++(-2.5pt,2.5pt);
\draw [color=rvwvcq, thick] (5,4) ++(-3.5pt,0 pt) -- ++(3.5pt,3.5pt)--++(3.5pt,-3.5pt)--++(-3.5pt,-3.5pt)--++(-3.5pt,3.5pt);
\draw[color=qqwuqq] (4.3,4.7) node {$c_x x + c_y y$};
\draw[color=black] (0.5,1.4) node {$d^\top y$};
\end{footnotesize}
\end{axis}
\end{tikzpicture}
\caption{Representation of the bilevel problem.}
\label{fig:bilevel102}
\end{minipage}
\begin{minipage}{0.75\textwidth}
\begin{tikzpicture}[line cap=round,line join=round,>=triangle 45,x=0.9cm,y=0.9cm]
\begin{axis}[
x=0.9cm,y=0.9cm,
axis lines=middle,
xmin=-0.5,
xmax=8.5,
ymin=-0.5,
ymax=5.5,
xtick={-1,0,...,14},
ytick={-3,-2,...,7},]
\clip(-1.0937275113352196,-3.747175320934984) rectangle (14.989133833193167,7.767575441284353);
\draw [line width=1pt] (2.5,0)-- (1,3);
\draw [line width=1pt] (2.5,0)-- (6,0);
\draw [line width=1pt] (6,0)-- (8,2.5);
\draw [line width=1pt,dash pattern=on 1pt off 2pt] (1,3)-- (5,4);
\draw [line width=1pt,dash pattern=on 1pt off 2pt] (8,2.5)-- (5,4);
\begin{footnotesize}
\draw [color=yqqqqq,very thick] (1,3)-- ++(-2.5pt,-2.5pt);
\draw [color=yqqqqq,very thick] (1,3)-- ++(2.5pt,2.5pt);
\draw [color=yqqqqq,very thick] (1,3)-- ++(2.5pt,-2.5pt);
\draw [color=yqqqqq,thick] (1,3)-- ++(-2.5pt,2.5pt);
\node[mark size=3pt,color=qqwuqq] at (5,0) {\pgfuseplotmark{*}};
\draw [line width=0.8pt,dotted,color=wwzzff,domain=-1.6578822632282277:11.84960867249812] plot(\x,{(--9--1*\x)/4});
\draw [line width=1pt,dash pattern=on 3pt off 2pt,color=ffwwqq,domain=-1.6578822632282277:11.84960867249812] plot(\x,{(--7--1*\x)/4});
\draw [line width=0.8pt,dotted,color=wwzzff,domain=-1.6578822632282277:11.84960867249812] plot(\x,{(--12-1*\x)/2});
\draw [line width=1pt,dash pattern=on 3pt off 2pt,color=ffwwqq,domain=-1.6578822632282277:11.84960867249812] plot(\x,{(--11-1*\x)/2});
\end{footnotesize}
\begin{scriptsize}
\draw[color=ffwwqq] (7.8,3.4) node[rotate=12] {$\delta = 1.0$};
\draw[color=wwzzff] (7.8,4.4) node[rotate=12] {$\delta = 0.5$};
\end{scriptsize}
\end{axis}
\end{tikzpicture}
\caption{Near-optimal robustness constraints.}
\label{fig:bilevel102-no}
\end{minipage}
\end{figure}

\noindent
The feasible space for $(\alpha,\beta)$ is given by:
\begin{align*}
&-1 \alpha_{11} - 4\alpha_{12} + \beta_{1} \geq 4\\
&-1 \alpha_{21} - 4\alpha_{22} + \beta_{2} \geq 2\\
&\alpha_{ki} \geq 0, \beta_{k} \geq 0.
\end{align*}

This feasible space can be described as a set of extreme points and rays.
It consists in this case of one extreme point
$(\overline{\alpha}_{ki} = 0, \overline{\beta}_1 = 4, \overline{\beta}_2 = 2)$ and 4 extreme rays.
The $(x,v)$ solution needs to be valid for the corresponding near-optimality conditions:
\begin{align*}
&\overline{\beta}_1 \,\, (v + \delta) \leq 11 + x \\
&\overline{\beta}_2 \,\, (v + \delta) \leq 13 - x.
\end{align*}

\noindent
This results in two constraints in the $(x,v)$ space, represented in \cref{fig:bilevel102-no}
for $\delta=0.5$ and $\delta=1.0$ in dotted blue and dashed orange respectively.
The radius of near-optimal feasibility $\hat{\delta} = 5$ can be computed using
the formulation provided in \cref{def:tolestim},
for which the feasible domain at the upper-level is reduced to the point $x = 5$,
for which $v=0$, represented as a green disk at $(5,0)$ in \cref{fig:bilevel102-no}.

\subsection{Valid inequalities}\label{sub:videf}

The extended formulation can be tackled directly in a branch-and-cut framework.
Nevertheless, we propose two groups of valid inequalities to tighten the formulation.\\

The first group of inequalities consists of the primal upper-level constraints:
\begin{align*}
& (G x)_k + (H v)_k \leq q_k & \forall k \in \left[\![m_u\right]\!].
\end{align*}

\noindent
These constraints are necessary for the optimistic formulation but not for the
near-optimal robust one since they are always redundant with and included
in the near-optimal robust constraints.
However, their addition can strengthen the linear relaxation of the extended formulation
and lead to faster convergence.\\

The second group of inequalities is defined in \cite{kleinert2020closing}
and based on strong duality of the lower level.
We only implement the valid inequalities for the root node, which are the primary focus of \cite{kleinert2020closing}:
\begin{equation}\label{eq:validineq}
\langle \lambda, b\rangle + \langle v, d\rangle \leq \langle A^{+}, \lambda\rangle,
\end{equation}
\noindent
where $A^{+}_i$ is an upper bound on $\langle A_i, x\rangle$.
The computation of each upper bound $A^{+}_i$ relies on solving an auxiliary problem:

\begin{subequations}
\begin{align}\label{prob:auxiliary}
A^+_i = \max_{x,v,\lambda}\,\, & \langle A_i, x \rangle\\
\text{s.t.} \,\,\,& G x + H v \leq q\\
& A x + B v\leq b\\
& d + B^\top \lambda \geq 0\\
& x \in \mathcal{X}, v \geq 0, \lambda \geq 0\\
& (x, v, \lambda) \in \Upsilon,\label{eq:jointvalid}
\end{align}
\end{subequations}

\noindent
where $\Upsilon$ is the set containing all valid inequalities \eqref{eq:validineq}.
\\

The method proposed in \cite{kleinert2020closing} relies on solving each $i$-th
auxiliary problem once and using the resulting bound $A^+$.
We define a new iterative procedure to improve the bounds computed at the root node,
similar to domain propagation techniques:

\begin{enumerate}
\item Solve Problem~\eqref{prob:auxiliary} $\forall i \in \left[\![m_l\right]\!]$ and obtain $A^{+}$;
\item If $\exists i,\ A^{+}_i$ is unbounded, terminate;
\item Otherwise, add Constraint~\eqref{eq:validineq} to \eqref{eq:jointvalid} and go to step 1;
\item Stopping criterion: when an iteration does not improve any of the bounds, terminate and return the last inequality with the sharpest bound.
\end{enumerate}

\noindent
This procedure allows tightening the bound as long as improvement can be made in
one of the $A^{+}_i$.
If the procedure terminates with one $A^{+}_i$ unbounded, the right-hand side of
\eqref{eq:validineq} is $+\infty$, the constraint is trivial and cannot be
improved upon. Otherwise, each iteration improves the bound until the convergence
of $A^{+}$.

\section{Solution algorithm}\label{sec:solalgs}

In this section, we describe the direct solution method for the
extended formulation, derive an exact variant lazifying
the subproblem exploration and provide a heuristic providing a
near-optimal robust solution.

Solving the extended formulation can be done by first optimizing the
dual adversarial subproblems, enumerating their vertices; if any subproblem
is infeasible, the procedure can be terminated as the instance cannot have a robust solution.

\subsection{Lazy subproblem expansion}

Directly solving the extended formulation of NRB quickly becomes computationally demanding when the problem size increases,
more lower-level constraints imply more vertices in each dual adversarial subproblem and thus larger disjunctive constraints,
more upper-level constraints imply more dual adversarial subproblems and thus more disjunctive constraints.
We next provide an exact method that builds upon this formulation while avoiding adding the whole set of disjunctive constraints upfront.

Each dual adversarial subproblem must be feasible
with an objective value of less than $q_k - G_k^\top x$ for the NRB instance to be near-optimal robust.
For $\mathcal{S} \subseteq \left[\![m_u\right]\!]$, we denote by $\overline{\text{NRB}}(\mathcal{S})$
the NRB formulation with the near-optimal robustness constraints replaced with:
\begin{align}
& G_k^\top x + H_k^\top v \leq q_k  & \forall z \in \mathcal{Z}(x;\delta)\,\, \forall k \in \mathcal{S}.\label{eq:nocons1}
\end{align}
$\overline{\text{NRB}}(\emptyset)$ does not have
any near-optimality robustness constraint and corresponds to the optimistic
bilevel problem, while
$\overline{\text{NRB}}(\left[\![m_u\right]\!])$ is equivalent to the extended formulation of NRB
and integrates all near-optimality robustness constraints.
$\overline{\text{NRB}}(\mathcal{S})$ is trivially a relaxation of
NRB since only a subset of the robust constraints is applied.\\

Furthermore, for a given pair $(\overline{x},\overline{v})$,
verifying its near-optimality robustness with respect to the
$k$-th upper-level constraint can be done with an auxiliary
linear optimization problem:
\begin{align}
\min_{\alpha,\beta}\,\, & (b - A \overline{x})^\top \alpha + (d^\top \overline{v} + \delta) \beta\label{prob:dualadvlazy} \\
\text{s.t.}\,\, & B^\top\alpha + \beta d \geq H_{k}\nonumber \\
& \alpha \in \mathbb{R}^{m_l}_{+} \beta \in \mathbb{R}_{+}.\nonumber
\end{align}
We consider a solution to the auxiliary problem as ``robust''
if the optimal value is below $q_k - G_k^\top \overline{x}$, in which case
the constraint is robust to near-optimal deviations of the lower level.
\Cref{alg:lazyenumeration} starts from $\mathcal{S} = \emptyset$
and iteratively adds violated constraints from \eqref{eq:nocons1}
that ensure the near-optimality robustness of some upper-level constraints.
We thus qualify it as lazy in contrast to the extended formulation
using all disjunctive constraints upfront.

\begin{algorithm}[!ht]
\caption{Lazy Subproblem Expansion}\label{alg:lazyenumeration}
\begin{algorithmic}[1]
\Function{LazySubProblemExpansion}{$G, H, q, c_x, c_y, A, B, b, d, \delta$}
	\State $\overline{\mathrm{NRB}}$ $\gets$ optimistic model with parameters $G, H, q, c_x, c_y, A, B, b, d$
	\If{Optimistic infeasible}
	\State \textbf{Terminate} with optimistic infeasible
	\EndIf
	\State $\hat{\mathcal{S}} \gets \left[\![m_u\right]\!]$
	\State $w_k \gets \{\}\,\, \forall k \in \left[\![m_u\right]\!]$
	\While{$\hat{\mathcal{S}}$ not empty}\label{algpoint:loop} \Comment{subproblem exploration phase}
	\State $\mathcal{S}_{opt} \gets \emptyset$
	\State $(\hat{x},\hat{v}) \gets$ current solution
	\State Choose $k \in \hat{\mathcal{S}}$
	\State \textbf{Solve} $k$-th dual subproblem~\eqref{prob:dualadvlazy} parameterized by $(\hat{x},\hat{v})$
	\If{Solution is robust}
		\State $\mathcal{S}_{opt} \gets \mathcal{S}_{opt} \cap \{k\} $
		\State $\hat{\mathcal{S}} \gets \hat{\mathcal{S}} \,\,\backslash\,\, \{k\} $
	\Else \Comment{subproblem expansion phase}\label{algpoint:expansion}
		\State Add new variables $w_{kl} \in \{0,1\} \,\,\forall l \in 1\dots|\mathcal{V}_k|$ to $\overline{\mathrm{NRB}}$
		\State Add constraint $\sum_{l=1}^{|\mathcal{V}_k|} w_{kl} \geq 1$ to $\overline{\mathrm{NRB}}$
		\For{$l \in 1\dots|\mathcal{V}_k|$}
		\State $(\overline{\alpha}_k^l, \overline{\beta}_k^l) \gets \mathcal{V}_k^l$
		\State Add indicator constraint to $\overline{\mathrm{NRB}}$:
		\State \,\,\,\,$w_{kl} = 1 \Rightarrow (b - Ax)^\top \overline{\alpha}_k^l + (d^\top v + \delta) \overline{\beta}_k^l \leq q_k - G_k^\top x$
		\EndFor
	\EndIf
	\State Solve current iterate $\overline{\mathrm{NRB}}$
	\If{Infeasible}
	\State \textbf{Terminate} with infeasible $k-$th subproblem
	\EndIf
	\State $\hat{\mathcal{S}} \gets \hat{\mathcal{S}} \cap \mathcal{S}_{opt} $
	\EndWhile
	\State \textbf{return} $(\hat{x}, \hat{v}, w)$
\EndFunction
\end{algorithmic}
\end{algorithm}

At any given iteration, the set $\hat{\mathcal{S}}$ is the complement
of $\mathcal{S}$,
\begin{equation*}
\hat{\mathcal{S}} \equiv \left[\![m_u\right]\!] \,\backslash\, \mathcal{S}
\end{equation*}
and $\mathcal{S}_{opt}$ contains the set of upper-level constraint indices
that are robust for a current iterate. If $\hat{\mathcal{S}}$ is empty,
there is no upper-level constraint that is not either added (already in $\mathcal{S}$)
or already optimal (in $\mathcal{S}_{opt}$). Given that $\overline{\mathrm{NRB}}$ is a relaxation
of NRB, $\hat{\mathcal{S}}=\emptyset$ implies that the optimum is reached.\\

At any given iterate of \cref{alg:lazyenumeration}, $\overline{\mathrm{NRB}}(\mathcal{S})$,
a MILP with disjunctive constraints, is solved from scratch.
The advantage over the extended formulation is that each of these MILPs is smaller and contains fewer
indicator constraints than the extended formulation.

\subsection{Single vertex heuristic}\label{sec:singlevtxheuristic}

We now present a heuristic method with a computational cost close to that of the canonical
bilevel problem and computing a high-quality bilevel-feasible near-optimal robust solution.
It only requires optimizing a sequence of at most $m_u$ MILPs with the same variables
as the canonical bilevel problem and at most $m_u$ additional linear constraints, instead
of the disjunctive constraints with a number of terms equal to the number of
vertices of the dual adversarial polyhedron.

\begin{algorithm}[!ht]
\caption{Single-Vertex Heuristic}\label{alg:heuristic}
\begin{algorithmic}[1]
\Function{SingleVertexHeuristic}{$G, H, q, c_x, c_y, A, B, b, d, \delta, \eta$}
	\State $P \gets$ optimistic model with parameters $(G, H, q, c_x, c_y, A, B, b, d)$
	\If{Optimistic infeasible}
	\State \textbf{Terminate} with optimistic infeasible
	\EndIf
	\If{$\exists k$, $k$-th dual adversarial subproblem is infeasible}\label{code:heuristicfeas}
	\State \textbf{Terminate}, no robust solution exists
	\EndIf
	\State $\text{count} \gets 1$
	\State $\mathcal{C} \gets \emptyset$
	\While{$\text{count} > 0$}\Comment{Outer loop} \label{code:outerloopheuristic}
	\State $k \gets 1$
	\State $\text{count} \gets 0$
	\State $(\hat{x},\hat{v}) \gets $ current solution of $P$
	\While{$\text{count} \leq \eta$ and $k \leq m_u$}
		\State $(\alpha, \beta) \gets$ solution to k-th subproblem \eqref{prob:dualadvlazy} parameterized by $(\hat{x},\hat{v})$ \label{code:lowerprob}
		\If{$(b - A \hat{x})^\top \alpha + (d^\top \hat{v} + \delta) \beta > q_k - G_k^\top \hat{x}$}
		\State $\mathcal{C} \gets \mathcal{C} \cup \{k\}$
		\State Add constraint to $P$: $(b - Ax)^\top \alpha + (d^\top v + \delta)\beta \leq q_k - G_k^\top x$
		\State $\text{count} \gets \text{count} + 1$
		\While{$k \in \mathcal{C}$}
			\State $k \gets k + 1$
		\EndWhile
		\EndIf
	\EndWhile
	\State Re-optimize $P$
	\If {Current iterate infeasible}
		\State \textbf{Terminate} with no found solution
	\EndIf
	\EndWhile
	\State \textbf{return} $(\hat{x}, \hat{v})$
\EndFunction
\end{algorithmic}
\end{algorithm}

At any given iterate, $\mathcal{C}$ is the set of upper-level constraint indices
that were already added to the model. New constraints are added in a batched
fashion, with the batch size controlled by the $\eta$ parameter.
Each $k$-th subproblem is added only once, by selecting a single vertex $(\alpha,\beta)$
and using it to enforce the constraint
\begin{equation*}
(b - Ax)^\top \alpha + (d^\top v + \delta) \beta \leq q_k - G_k^\top x.
\end{equation*}
Unlike the exact algorithms, \cref{alg:heuristic} initializes a MILP model
and iteratively adds linear constraints to it. This procedure therefore lends
itself to warm starts and single-tree formulations.

\begin{proposition}
\cref{alg:heuristic} terminates in at most $m_u$ iterations of the outer loop
beginning \cref{code:outerloopheuristic} and solves optimization problems with the same
variables as the canonical optimization problem and at most $m_u$ linear constraints.
\end{proposition}

\begin{proof}
\cref{alg:heuristic} maintains a cache $\mathcal{C}$ of the subproblems that
have been explored.
For each subproblem, exactly one vertex is chosen, which minimizes
\begin{equation*}
  (b - A\hat{x})^\top  \alpha + (d^\top \hat{v} + \delta) \beta
\end{equation*}
with $(\hat{x}, \hat{v})$ the current iterate. For a chosen $\hat{x}$,
the lower-level problem is feasible since $\hat{v}$ is computed, so the dual
problem cannot be unbounded. It cannot be infeasible since its feasibility domain
depends only on the problem data and is verified \cref{code:heuristicfeas}.
Therefore, a vertex $(\alpha,\beta)_k$ is computed. If the objective is greater
than $q_k - G_k^\top \hat{x}$, then the current iterate is not near-optimal robust
with respect to the $k$-th constraint, and the linear constraint:
\begin{equation*}
  (b - Ax)^\top \alpha + (d^\top v + \delta) \beta \leq q_k - G_k^\top x
\end{equation*}
is added to the model. If no addition is made, the current iterate is
near-optimal robust, the count variable remains at 0, and the outer loop exits,
with the function returning the current iterate.
Otherwise, the iterate was not near-optimal robust with respect to at least one
upper-level constraint, which is turned into a constraint and added to the cache.
The outer loop adds at least one constraint in a non-terminating iteration,
therefore, $m_u$ iterations suffice to add all constraints.
Moreover, each subproblem adds exactly one linear constraint, so at most $m_u$
linear constraints are added to the formulation of the canonical problem.
\end{proof}

The $\eta$ parameter controls the maximum number of linear constraints added for
each outer iteration. $\eta = 1$ implies that the algorithm reoptimizes the
problem after adding a single constraint, while $\eta \geq m_u$ will add all the
linear constraints that correspond to upper-level constraints that are not
near-robust at each iterate.\\

Finally, it can be noted that the single-vertex algorithm can be applied even
when the number of vertices in the dual adversarial is infinite, i.e.\ when
the lower-level problem is a convex optimization problem.
The only modification is the optimization of the dual adversarial problem for
fixed values of $(x,v)$ at \cref{code:lowerprob}, where a convex optimization
problem is solved instead of a linear one.

\section{Computational experiments}\label{sec:numerics}

In this section, we demonstrate the applicability of our approach through numerical
experiments on instances of the linear-linear near-optimal robust bilevel problem.
We first describe the sets of test instances and the computational setup and then
the experiments and their results.

\subsection{Instance sets}

Two sets of data are considered.
For the first one, a total number of 1000 small, 200 medium and 100 large random
instances are generated and characterized as follows:
\begin{align*}
&(m_u, m_l, n_l, n_u) = (5,5,5,5) && \text{(small)}\\
&(m_u, m_l, n_l, n_u) = (10, 10, 10, 10) && \text{(medium)}\\
&(m_u, m_l, n_l, n_u) = (20, 10, 20, 20) && \text{(large)}.
\end{align*}

All matrices are randomly generated with each coefficient having a 0.6
probability of being 0 and uniformly distributed on $\left[0,1\right]$
otherwise.
High-point feasibility and the vertex enumeration procedures are run after
generating each tuple of random parameters to discard infeasible instances.
Collecting 1000 small instances required generating 10532 trials,
the 200 medium-sized instances were obtained with 18040 trials and the 100 large
instances after 90855 trials.
A second dataset is created from the 50 MIPS/Random instances of the Bilevel
Problem library \cite{blpinstances}, where integrality constraints are dropped.
All of these instances contain 20 lower-level constraints and no upper-level constraints.
For each of them, two new instances are built by moving either the first 6 or the last 6
constraints from the lower to the upper level, resulting in 100 instances.
We will refer to the first set of instances as the small/medium/large instances
and the second as the MIPS instances.
All instances are available in \cite{besancon_mathieu_2020_4009109} in JLD format,
along with a reader to import them in Julia programs.

\subsection{Computational setup}

All experiments are carried out in Julia 1.6 \cite{bezansonjulia2017}
using the \texttt{JuMP} v0.21 modelling framework \cite{DunningHuchetteLubin2017,legat2020mathoptinterface};
the MILP solver is \texttt{SCIP} 7.0 \cite{scip} with \texttt{SoPlex} 5.0 as the inner LP solver, both with default
solving parameters. \texttt{SCIP} handles indicator constraints in the form of linear
inequality constraints activated only if a binary variable is equal to one.
\texttt{Polyhedra.jl} \cite{juliapolyhedra} is used to model the dual subproblem
polyhedra with \texttt{CDDLib} \cite{CDDLib} as a solver running the double-description
algorithm, computing the list of extreme vertices and rays from the
constraint-based representation.
The exact rational representation of numbers is used in \texttt{CDDLib} instead of
floating-point types to avoid rounding errors.
Moreover, \texttt{CDDLib} fails to produce the list of vertices for some instances
when set in floating-point mode.
All experiments are performed on a workstation with 32GB of RAM and
Intel Xeon 3.5GHz CPUs running Ubuntu 18.04LTS.

\subsection{Bilinear and extended formulation}

To assess the efficiency of the extended formulation, we compare its solution time
to that of the non-extended formulation including bilinear constraints
\eqref{prob:lin101}.
The bilinear formulation is implemented with \texttt{SCIP} using \texttt{SoPlex} as the linear
optimization solver and \texttt{Ipopt} as the nonlinear solver.
\texttt{SCIP} handles the bilinear terms through bound computations and spatial branching.
Out of all MIPS and ALTMIPS sets, only one instance is solved with the bilinear
formulation within the time limit in half a second, a time similar to the extended formulation.
The bilinear formulation not only runs out of time, but also of memory
(a memory limit of 5000MB was fixed in this setting).
The spatial branching required to handle the non-convex bilinear inequalities
is thus more time- and memory-consuming than the branching over disjunctive constraints
introduced by the extended formulation. Because of the vertices of the dual
adversarial subproblems being optimal solutions as developed in \cref{sub:extended},
the disjunctive constraints explicitly constrain the optimality candidates to a discrete set
instead of exploring a continuous set of $(\alpha,\beta)$ solutions through spatial branching. 

\subsection{Robustness of optimistic solutions and influence of $\delta$}

We solve the MIPS instances to optimistic bilevel optimality and verify the near-optimal
robustness of the obtained solutions. We use various tolerance values:
\begin{equation*}
\delta = \text{max}(0.05, \delta_r \times \text{opt}(L))
\end{equation*}
with $\text{opt}(L)$ the lower-level objective value at the obtained optimistic solution and
\begin{equation*}
\delta_r \in \{0.01, 0.05, 0.1, 0.5, 3.0\}.
\end{equation*}
Out of the 100 instances, 57 have canonical solutions that are not
robust to even the smallest near-optimal deviation $0.01 \text{opt}(L)$.
Twelve more instances that have a near-optimal robust solution with the lowest tolerance
are not near-optimal robust when the tolerance is increased to $3\text{opt}(L)$.
Out of the 57 instances that are not near-optimal robust with the lowest tolerance,
40 have exactly one upper-level constraint that is violated by near-optimal
deviations of the lower level and 17 that have more than one.
Finally, we observe that the number of violated constraints changes across the
range of tolerance values for 31 out of 100 instances. For the other 69 instances,
the number of violated upper-level constraints remains identical for all tolerance values.





\subsection{Computation time of the different algorithms}

Statistics on the computation times of the vertex enumeration and solution phases for the
extended formulation on each set of instances are provided in \cref{tab:vertextime} and \cref{tab:comptime}.

\begin{table}[ht]
\centering
\begin{tabular}{|c|c|c|c|c|}
\hline
Set & mean & $10\%$ quant. & $50\%$ quant. & $90\%$ quant. \\
\hline
MIPS & 15.79 & 0.15 & 2.38 & 53.37\\
\hline
ALTMIPS & 0.17 & 0.05 & 0.13 & 0.42\\
\hline
\end{tabular}
\caption{Runtime statistics for the vertex enumeration (s).}
\label{tab:vertextime}
\end{table}

\begin{table}[ht]
\centering
\begin{tabular}{|c|c|c|c|c|c|}
\hline
Set & $\#$ optimized & mean & $10\%$ quant. & $50\%$ quant. & $90\%$ quant. \\
\hline
MIPS & 47 & 474.657 & 2.583 & 71.219 & 1659.787\\
\hline
ALTMIPS & 41 & 603.040 & 1.392 & 66.819 & 1751.774\\
\hline
\end{tabular}
\caption{Runtime statistics for the optimization phase (s).}
\label{tab:comptime}
\end{table}

The solution time is greater than the vertex enumeration phase which is thus not the bottleneck to solve NRB instances.\\

We also studied the sensitivity of NRB solutions to variations of $\delta$ on small-size random instances.
When $\delta$ increases, the number of problems solved to optimality monotonically decreases;
greater $\delta$ values indeed reduce the set of feasible solutions to {NRB}.
The optimal values of the upper-level only slightly increase with $\delta$
and the lower-level objective value does
not vary significantly with $\delta$.
\par Even though more instances become infeasible as $\delta$ increases,
the degradation of the objective value is in general insignificant for the
optimal near-optimal robust solution compared to the optimistic solution.\\


\begin{figure}[ht]
\centering
\includegraphics[width=0.8\textwidth]{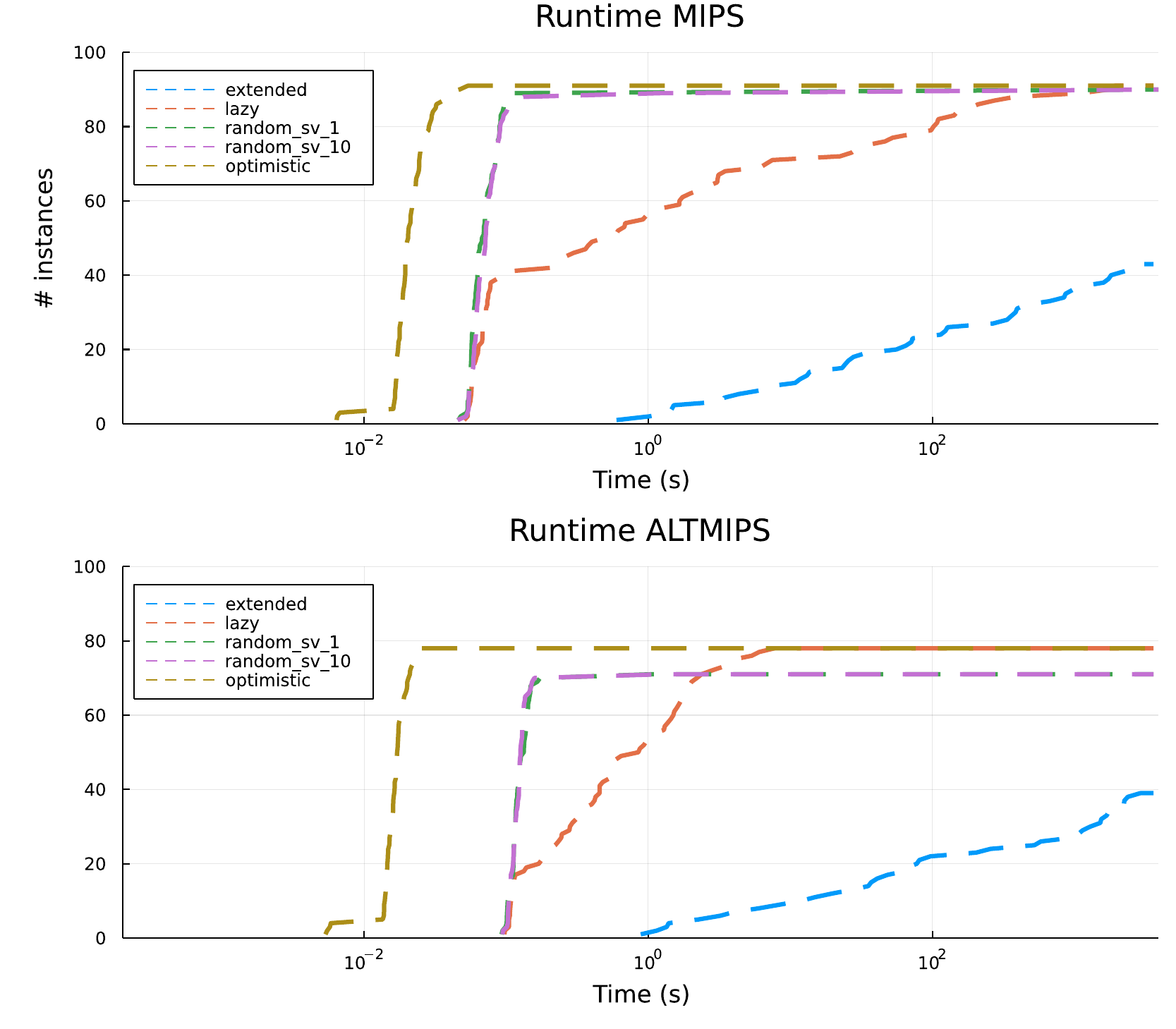}
\caption{Comparing the different solution methods on the two instance sets.}
\label{fig:comparisonall}
\end{figure}

The different solution methods are compared in \cref{fig:comparisonall}.
The single vertex heuristics with various batch sizes, labelled
\texttt{random\_sv\_batchsize}, are the fastest to compute near-optimal robust
solutions, slightly slower than solving the optimistic bilevel formulation.
On exact methods, the \texttt{lazy} vertex enumeration \cref{alg:lazyenumeration}
outperforms the eager \texttt{extended} formulation
and terminates within the time limit for all instances.
One MIPS instance is bilevel feasible but does not possess a near-optimal robust
solution. On 90 MIPS and 71 ALTMIPS instances, the single vertex heuristic successfully finds
a solution while it terminates unsuccessfully in one and seven cases respectively.
Furthermore, the heuristic runtime is fast enough compared to the exact methods
that it could be improved by randomly removing some vertices that were added in the
first iterations to then possibly reintroduce them to the formulation.
Finally, we note that when the single-vertex heuristic finds a solution,
it is one of high quality. On no MIPS instance and only two ALTMIPS instances
is the objective value of the heuristic solution not equal to that of the
exact method, verified with the lazy algorithm.

\subsection{Implementation of valid inequalities}

In the last group of experiments, we implement and investigate the impact of
the valid inequalities defined in \cref{sub:videf}.\\

The valid inequality procedure found cuts for all 100 MIPS instances.
For 98 of these, the procedure terminates after one iteration, the two other
instances terminate with a cut after 4 and 8 iterations.
On the 100 ALTMIPS instances, 18 are bilevel-infeasible,
none of which had an infeasible high-point relaxation.
For 12 of these instances, adding the valid inequalities was
enough to prove their infeasibility.
For all instances, the procedure computed non-trivial valid inequalities
i.e.~all coefficients of $A^+$ are finite.

These results highlight the improvement of the model tightness with the addition of
the valid inequalities, compared to the high-point relaxation where
primal and dual variables are subject to distinct groups of constraints.
These inequalities thus discard infeasible instances without the need to solve
a MILP.
For all but two MIPS instances, a single iteration of the procedure computes
the final valid inequality.
12 out of 100 ALTMIPS instances require more than one iteration
with one instance requiring up to 32 iterations.

We also compared the total runtime for MIPS and ALTMIPS instances under near-optimality robustness constraints using $\delta=0.1$
with and without valid inequalities for all instances solved to optimality.
Valid inequalities do not improve the runtime for NRB in either group of instances,
This result is similar to the observations in \cite{kleinert2020closing} for instances
of the canonical bilevel linear problem without near-optimality robustness.\\

We next study the inequalities based on the upper-level constraints on the small,
medium and MIPS instances.

\begin{figure}[ht]
\includegraphics[width=0.8\textwidth]{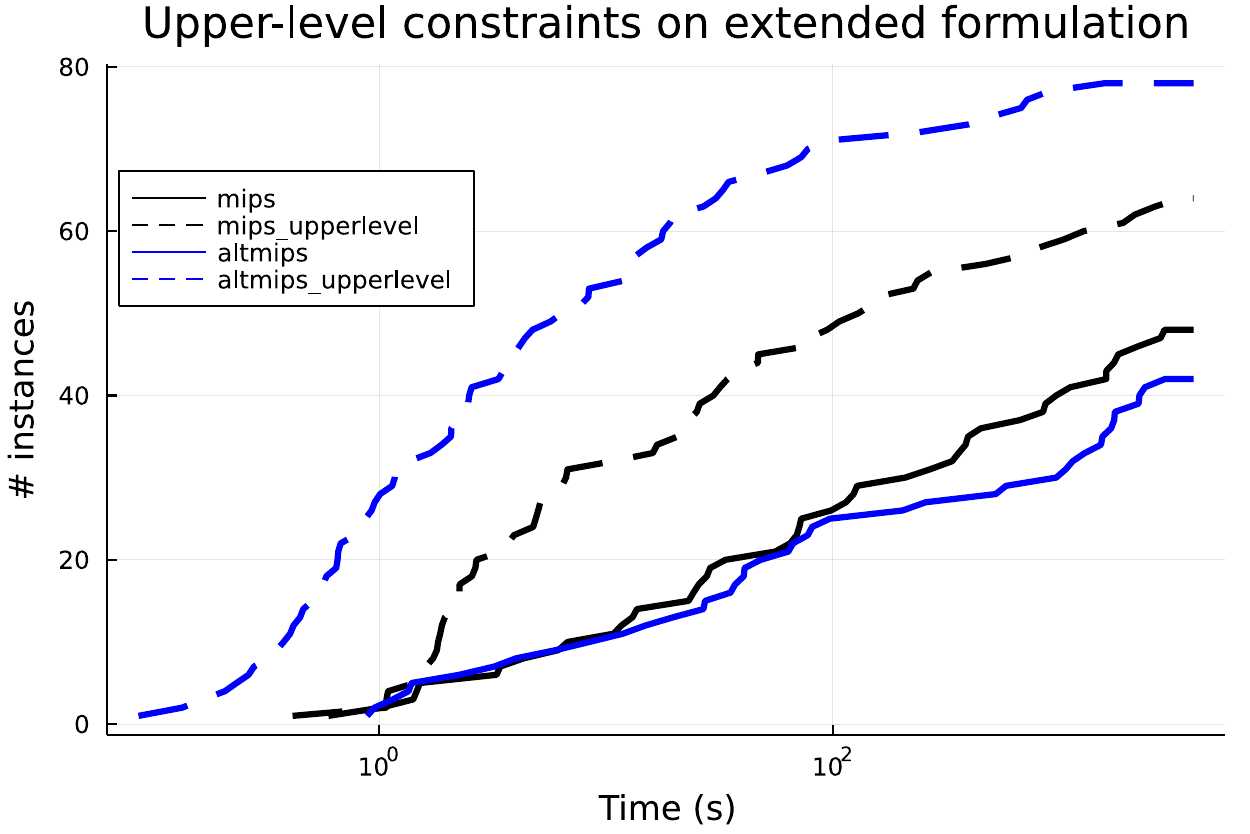}
\caption{Runtime for MIPS and ALTMIPS instances with and without upper-level constraints.}
\label{fig:upperlevelcons}
\end{figure}

\noindent
As shown in \cref{fig:upperlevelcons}, the addition of primal upper-level
constraints accelerates the resolution of the MIPS and even more of
ALTMIPS instances and dominates the standard extended formulation.

\section{Conclusion}\label{sec:conclusionNRB}

In this work, we introduce near-optimal robust bilevel optimization, a specific
formulation of the bilevel optimization problem 
does not make an explicit optimistic/pessimistic assumption about the reaction of the lower level, but instead seeks an optimal decision at the upper level that is  robust in the sense that it remains feasible even if the lower level deviates from its own optimality within a prescribed near-optimal set for the lower level.
Near-optimality robustness challenges the assumption that the lower-level problem
is solved to optimality, resulting in a generalized, more conservative formulation
including the optimistic and pessimistic bilevel problems as special cases.
We formulate NRB in the dependent case, i.e.\ where the upper- and lower-level
constraints depend on both upper- and lower-level variables, thus
offering a framework applicable to many bilevel problems of practical interest.
\\

We derive a closed-form, single-level expression of NRB for convex
lower-level problems, based on dual adversarial certificates to guarantee near-optimality robustness.
In the linear case, we derive an extended formulation
that can be represented as a MILP with indicator constraints.
This extended formulation lends itself to a faster exact lazy method and
a single-vertex heuristic leveraging the problem structure to find high-quality
feasible solutions.
Numerical experiments highlight the efficiency of the extended method compared
to the compact bilinear formulation, the impact of valid inequalities on
both solution time and tightness of the linear relaxation.
Finally, they highlight the interest of the lazy and heuristic formulation,
showing a solution time only slightly slower than the corresponding 


\bibliographystyle{ieeetr}
\bibliography{refs}

%


\end{document}